\documentclass[pdflatex,sn-mathphys-num]{sn-jnl}

\usepackage{graphicx}%
\usepackage{multirow}%
\usepackage[T2A]{fontenc}
\usepackage[russian,french, main=english]{babel}  
\usepackage{amsmath,amssymb,amsfonts}%
\usepackage{amsthm}%
\usepackage[title]{appendix}%
\usepackage{xcolor}%
\usepackage{textcomp}%
\usepackage{manyfoot}%
\usepackage{booktabs}%
\usepackage{algorithm}%
\usepackage{algorithmicx}%
\usepackage{algpseudocode}%
\usepackage{listings}%
\usepackage{mathtools}

\usepackage{hyperref}
\usepackage{tabularx}
\usepackage{seqsplit} 
\usepackage{tabularx}
\usepackage{booktabs}
\usepackage{seqsplit}
\usepackage{float}
\usepackage{array}
\usepackage{longtable}
\usepackage[table]{xcolor}
\definecolor{highlight}{RGB}{255, 191, 0}

\newtheorem{theorem}{Theorem}
\newtheorem{lemma}{Lemma}
\newtheorem{definition}{Definition}
\newtheorem{corollary}{Corollary}
\newtheorem{conjecture}{Conjecture}

\newtheorem{Proper}{Properties}

\theoremstyle{thmstyleone}%
\newtheorem{proposition}[theorem]{Proposition}%

\theoremstyle{thmstyletwo}%
\newtheorem{example}{Example}%
\newtheorem{remark}{Remark}%
\numberwithin{equation}{section}

\newcommand{\bk}{\mathbf{k}}

\newcommand{\ZZ}{\mathcal{Z}}

\newcommand{\zA}{\zeta^{}_\A}

\newcommand{\BZ}{\mathbb{Z}}

\newcommand{\BN}{\mathbb{N}}

\newcommand{\A}{\mathcal{A}}
\newcommand{\Z}{\mathbb{Z}}
\newcommand{\Q}{\mathbb{Q}}
\newcommand{\R}{\mathbb{R}}

\theoremstyle{thmstylethree}%
\newtheorem{question}{Question}%

\newcommand{\gaW}{\gamma_\A^\mathrm{W}}
\newcommand{\gaG}{\gamma_\A^\mathbb{G}}
\newcommand{\gaAG}{\gamma_\A^\mathbb{AG}}
\newcommand{\gaQ}{\gamma_\A^\mathrm{Q}}
\newcommand{\gaL}{\gamma_\A^\mathrm{L}}
\newcommand{\gaM}{\gamma_\A^\mathrm{M}}
\newcommand{\gaKp}{\gamma_\A^\mathrm{K_p}}

\def\la{\ell_\A}

\def\={\,=\,}

\begin{document}

\title[Gertsch] 
{Gertsch quotient living in the ``poor man’s adele ring" $\A$: Kurepa-Bell-Wilson congruence}


\author[]{\fnm{Francis Atta} \sur{Howard}}\email{hfrancisatta@ymail.com; hfrancisatta@gmail.com}



\affil[]{\orgdiv{International Chair in Mathematical Physics and Applications(CIPMA-UNESCO)}, \orgname{University of Abomey-Calavi}, \city{Cotonou}, \postcode{072 B.P. 50}, \country{Benin Republic}}



\abstract{Wilson's theorem is notably related to left factorials, expressed as $K_p \equiv \mathbf{Bell}_{p-1} - 1 \pmod p$, for prime $p\geq3$. This study examines a Kurepa-Bell-Wilson congruence (\textbf{KBW}), $\frac{K_p + 1}{p}\equiv \frac{ \mathbf{Bell}_{p-1}}{p}+ W_p \pmod{p}$, and demonstrates that it naturally generates the non-zero "Gertsch quotient ($\mathbb{G}_p$)," which, for larger primes modulo $p$ resides in the poor man's adele ring $\A$ .}




\maketitle
\tableofcontents
\section{Introduction}\label{sec1}

The Left factorials $!n$ have drawn much attention since 1971 when introduced by Duro Kurepa \cite{kurepa1971left} as $!n=\sum_{m=1}^{n-1} m!=0! +1! + 2! +\ldots + (n-1)!$ for all $n\in \BN$. Kurepa conjectured that the $\gcd(!n, n!)=2$ for $n>2$ and added the equivalent that this conjecture is the same as checking for every odd prime $p>3$, $!p\not\equiv 0 \pmod{p}$. The evidence of interesting investigations has presented several approaches to solving this by trying to prove both cases. Authors including Mijajlović, Barsky, Petojević, Ivić, Gertsch, Zižović, Howard, Andrejić and many others \cite{barsky2004nombres, zivkovic1999number, mijajlovic1990some, gertsch1999congruences, howard2025partition, andrejic2016searching} actively investigate rigorous approach to prove this conjecture, recently, the author developed a new equivalent to the Kurepa conjecture which he solved by using techniques in Binary $\mathbf{GCD}$ and other tools to prove that the  $\gcd(\mathbb{F}_{n},(n+1)!)=2$ for $n>1.$ Several computational approach were used by authors to try to resolve the conjecture by either checking if there exists a counterexample or otherwise. In the quest to prove the Kurepa conjecture in terms of prime $p>3$, certain authors proposed new equivalents in terms of modulo $p$ to solve $!p\not\equiv 0 \pmod{p}$, in particular, Gertsch \cite{gertsch1999congruences} in 1999 proposed in her thesis "congruences pour quelques suites classiques de nombres; sommes de factorielles et calcul ombral" that 
$$   !p=K_p=\sum_{ n=0}^{p-1} n!= \sum_{ n=0}^{p-2} n! + (p-1)!\equiv \mathbf{Bell}_{p-1} -1 \pmod{p}$$
this new insight triggered some researches such as Daniel Barsky \cite{barsky2004nombres} to investigate $K_p\not\equiv 0 \pmod{p}$ as well as $\mathbf{Bell}_{p-1}\not\equiv 1 \pmod{p}$.
Gertsch's thesis is a great source of reference that brings indepth details and understanding of the Kurepa factorials in terms of modulo prime. 
 Mijajlović also proposed a new equivalent in terms of derangement modulo prime, specifically, in a Galois field of prime $p>3$;
 $\mathbf{Der}_{p-1}= \sum_{k=0}^{p-1} (-1)^{k+1}\frac{1}{k!} $ which is just  $K_p\equiv \mathbf{Der}_{p-1} \pmod{p} $ as proposed by Gertsch. This elegant connection makes researchers such as Don Zagier and Sun \cite{sun2011curious} investigate the possible nexus between Bell modulo prime and the derangement numbers from the identity 
 $$  \mathbf{Der}_{p-1} +1 \pmod{p}\equiv\mathbf{Bell}_{p-1}. $$
In 2002, V. S. Vladimirov gave a new equivalent to the Kurepa conjecture which involved the p-adic numbers and Bernoulli numbers, that is, 
$$!p \sum_{k=0}^{p-2} (-1)^k \frac{\mathbf{B}_k}{k!} \equiv \sum_{m=1}^{(p-3)/2} \frac{\mathbf{B}_{2m}}{(2m)!} [!(2m)-1] \pmod{p}$$
where $\mathbf{B}_k$ is the $k$-th Bernoulli number. From this relation, He gave a sufficient condition for the Kurerpa conjecture to be true. 
The Wilson's theorem is evident in the Kurepa modulo $p$ since the $(p-1)!\equiv -1 \pmod{p}$, furthermore, great mathematicians such as Glaisher, Lehmer, Carlitz, Niel, Norlund and many others have shown how the Wilson's theorem connect to the Bernoulli numbers modulo prime \cite{nielsen1913recherches, carlitz1961staudt, glaisher1900residues, beeger1913quelques}. Guiga \cite{borwein1996giuga} established a much beauitiful sum $s_n$ that links to our understanding of the primes and also Agoh \cite{agoh1995giuga} conjectured that 
 $$p\mathbf{B}_{p-1}\equiv -1 \pmod{p}$$
and also investigated various kinds of congruences related to the Wilson quotient by applying a Miki-type linear identity involving two different kinds of sums for Bernoulli numbers \cite{agoh2022congruences}. 

Let $\A$ be the quotient ring \cite{kontsevich2009holonomic} 
$$\A\coloneqq\dfrac{\left(\prod_p \Z/p\Z\right)}{\left(\bigoplus_p \Z/p\Z\right)}.$$
of the direct product of $\Z/p\Z$ over all primes modulo the ideal of the direct sum.
The field $\Q$ of rational numbers can naturally be embedded diagonally in $\A$,
and via this the ring $\A$ is regarded as a $\Q$-algebra.
Kaneko and Zagier \cite{kaneko2017finite, kaneko2019analogues, kaneko2021finite} jointly discovered and defined for each tuple of positive integers $(k_1,\ldots, k_r)$, a finite multiple zeta value $\zA(k_1,\ldots, k_r)$ in $\A.$ 
They conjectured fundamentally the existence of an isomorphism of
$\Q$-algebras between $\ZZ_\A$ and $\ZZ_\R/\zeta(2)\ZZ_\R$, where $\ZZ_\R$ is the $\Q$-algebra of usual multiple 
zeta values in $\R$ and $\zeta(2)\ZZ_\R$ is the ideal of $\ZZ_\R$ generated by $\zeta(2)$ 
with explicitly defined element $\zeta^{}_\mathcal{S}(\bk)$ in $\ZZ_\R/\zeta(2)\ZZ_\R$
which conjecturally corresponds to $\zA(\bk)$ under the predicted isomorphism. See \cite{kaneko2019analogues, kaneko2017finite} for more details, and \cite{zhao2016multiple}
for multiple zeta values in general.
For the Bernoulli number $\mathbf{B}_n$  and for each $k\ge2$ an element $Z_\A(k)$ in $\A$ given by 
\[ Z_\A(k)\coloneqq\left(\frac{\mathbf{B}_{p-k}}k\,\bmod p\right)_p\]
for more details on this, see \cite{ kontsevich2009holonomic, murahara2016note, kaneko2017finite, kaneko2025finite}.\\
The remainder of the study is organized as follows: 
Section \ref{sec2} deals with the main results of this study. We revisit already known concepts about Kurepa and Bell modulo prime following Gertsch's Thesis in section \ref{sec 3}. In addition, we consider some quotients involving Fermat, Wilson, Lerch and then introduce a quotient called the "Gertsch quotient" (c.f \cite{OEIS:A309483}) from Kurepa modulo prime in section \ref{sec 4}. In section \ref{Bernoulli}, we recall Vladimirov's theorem \cite{TA} on Kurepa-Bernoulli modulo prime and extend result to Wilson quotient, we also do same for the genus of Hodge integral $b_g$ \cite{faber1998hodge, faber2011tautological}. In section \ref{sec 7} we look at some known results about the Euler constant following Kaneko et al. \cite{kaneko2025finite, kaneko2017finite, kanekoZagier2026} and then extend these results to Gertsch quotient, Lerch quotient and Agoh-Guiga quotient. We finally end with some numerical Heuristics and the "congrunce $0\bmod{p}$" in section \ref{sec 8}.


\section{Main results}\label{sec2}
\begin{theorem}\label{Kaneko-ring}
Let $\A\coloneqq\left. \left(\prod_p \Z/p\Z\right) \middle/ \left(\bigoplus_p \Z/p\Z\right) \right.$ be the poor man's adele ring and $K_p= \sum_{ n=0}^{p-1} n!$ be the Kurepa factorial for odd primes $p\geq3$ then,
    $$ \gaKp\coloneqq\left(   K_p \bmod p\right)_p=\left(   \mathbf{Der}_{p-1} \bmod p\right)_p \in \A,$$
    specifically, $\gaKp\neq 0.$
\end{theorem}

\begin{theorem}\label{atta}
Let  $K_p \equiv \mathbf{Bell}_{p-1} -1 \pmod p$ be the Kurepa modulo prime $p\geq 3$ and $W_p = \frac{(p-1)! + 1}{p}$ be the Wilson quotient;  then  
\begin{align*}\label{KBW}
    \frac{K_p + 1}{p}\!\underset{\text{Wilson}}{\equiv} \frac{ \mathbf{Bell}_{p-1}}{p}+ W_p\tag{\textbf{KBW}}
\end{align*}  
and
\begin{align*}\label{KBL}
\frac{K_p + 1}{p}\!\underset{\text{Lerch}}{\equiv}\frac{ \mathbf{Bell}_{p-1}}{p}+ \sum_{x=1}^{p-1} q_p(x)\tag{\textbf{KBL}}
\end{align*}
if we set $x=2$ and $x=3$ then,

\begin{align*}\label{KBF}
\frac{K_p + 1}{p}\!\underset{\text{Fermat}}{\equiv}\frac{ \mathbf{Bell}_{p-1}}{p}-q_p(2)\tag{\textbf{KBF}},
\end{align*} 

\begin{align*}\label{KBM}
\frac{K_p + 1}{p}\!\underset{\text{Mirimanoff}}{\equiv}\frac{ \mathbf{Bell}_{p-1}}{p}-q_p(3)\tag{\textbf{KBM}}.
\end{align*} 

From equations \ref{KBW}, \ref{KBL}, \ref{KBF} and definition \ref{Euler-W} with $\gaW, \quad \log_\A(x)\in \A$, then

\begin{align}\label{E-wil}
    \left( \frac{K_p + 1}{p}\bmod p\right)_p\!\underset{\text{Wilson}}{=} \left(\frac{ \mathbf{Bell}_{p-1}}{p}\bmod p\right)_p + \gaW
\end{align}  
and
$$ \left( \frac{K_p + 1}{p}\bmod p\right)_p\!\underset{\text{Lerch}}{=}\left(\frac{ \mathbf{Bell}_{p-1}}{p}\bmod p\right)_p+ \sum_{x=1}^{p-1}\log_\A(x) $$
if $x=2$ and $x=3$ then
$$ \left( \frac{K_p + 1}{p}\bmod p\right)_p\!\underset{\text{Weiferich}}{=}\left(\frac{ \mathbf{Bell}_{p-1}}{p}\bmod p\right)_p -\log_\A(2), $$

$$ \left( \frac{K_p + 1}{p}\bmod p\right)_p\!\underset{\text{Mirimanoff}}{=}\left(\frac{ \mathbf{Bell}_{p-1}}{p}\bmod p\right)_p -\log_\A(3), $$

where $\A\coloneqq\left. \left(\prod_p \Z/p\Z\right) \middle/ \left(\bigoplus_p \Z/p\Z\right) \right.$.
\end{theorem}

\begin{proposition}\label{strong cond}
Let  $K_p\not \equiv 0 \bmod{p}$ be the Kurepa conjecture.
    If there exists odd primes \( p \geq 2 \times 10^{13} \) for which \( W_p \equiv 0 \mod{p} \), such that \( W_p \equiv \mathbb{G}_p \mod{p} \) and satisfies the Gertsch quotient $$ \mathbb{G}_p \equiv 0 \mod{p} $$ then the Kurepa conjecture is untrue.
\end{proposition}

\begin{corollary}\label{st cond}

    The Kurepa conjecture, $K_p \not\equiv 0 \mod{p}$, is valid if there are no odd primes $p\geq 3$ such that $\mathbb{G}_p \equiv 0 \mod{p}$.
\end{corollary}

\begin{corollary}\label{st cnd}
    Any odd prime $p\geq 2\times 10^{13}$, for which the Gertsch quotient, $\mathbb{G}_p \equiv 0 \mod{p}$ and $W_p\equiv  \mathbb{G}_p \mod{p}$ are satisfied, is also a Wilson prime.
\end{corollary}

\section*{Notation and terminology}
In order to avoid ambiguity due to similar notations in literature we shall use $K_p$ to represent the Kurepa modulo prime instead of $!p$, we shall use $K_n$ to replace the left factorials $!n$ in most cases. The bold face symbols, $\mathbf{Bell}_{n}$ represents Bell numbers, $\mathbf{B}_{k}$ is the $k$-th Bernoulli numbers, $\mathbf{Der}_k$ denotes the derangement numbers. The Gertsch quotient is denoted by $\mathbb{G}_p$, the Vladimirov numbers by the atoms $V_p$ and $V_{p}^{*}$. Following the Agoh and Giuga conjecture we shall denote $\mathbb{AG}_p$ as Agoh-Giuga quotient. $\mathbf{KBW}$ is Kurepa-Bell-Wilson, $\mathbf{KBL}$ is Kurepa-Bell-Lerch, and $\mathbf{KBF}$ is Kurepa-Bell-Fermat.  

\section{Gertsch's results on Kurepa and Bell modulo p}\label{sec 3}

\begin{conjecture}\cite{gertsch1999congruences, kurepa1971left}
    Let $p$ be an odd prime, then 
    \begin{enumerate}
        \item $K_p=\sum_{0\leq n\leq (p-1)} n!= 0! + 1! + 2! + \cdots + (p-2)! + (p-1)!$ is not congruent to zero modulo $p.$ 
    \item the $(p-1)^{th}$ Bell number $\mathbf{Bell}_{p-1}$ is not congruent to 1 modulo $p.$
     \end{enumerate}
\end{conjecture}

The table \ref{kp_bell} below gives a simple view of the connection between the Kurepa modulo prime and the Bell numbers, specifically, we shall see in the subsequent sections how this connection easily builds up.

\begin{table}[h]
\centering
\caption{Values for $K_p$ and $\mathbf{Bell}_{p-1}$ for primes $p\leq 17$ \cite{OEIS:A003422}}
\label{kp_bell}
\begin{tabular}{|c|r|c|r|c|}
\hline
$p$ & $K_p$ & $K_p \pmod{p}$ & $\mathbf{Bell}_{p-1}$ & $\mathbf{Bell}_{p-1} \pmod{p}$ \\
\hline
3 & 4 & 1 & 2 & 2 \\
5 & 34 & 4 & 15 & 0 \\
7 & 874 & 6 & 203 & 0 \\
11 & 4\,037\,914 & 1 & 115\,975 & 2 \\
13 & 522\,956\,314 & 10 & 4\,213\,597 & 11 \\
17 & 22\,324\,392\,524\,314 & 13 & 10\,480\,142\,147 & 14 \\
\hline
\end{tabular}
\end{table}

\begin{Proper}\label{Gt}\cite{gertsch1996some}
    For $n=p-1$, with an odd prime $p$, the following congruence modulo $p$ hold: 
   \begin{enumerate}
       \item $\mathbf{Der}_{p-1}\equiv K_p \mod{p}$
       \item $\mathbf{Der}_{p-1}\not\equiv  0 \mod{p}$  is equivalent the Kurepa conjecture
       \item $\mathbf{Bell}_{p-1}\equiv \mathbf{Der}_{p-1} + 1\mod{p}$
       \item $K_p\equiv \sum_{0\leq n\leq (p-1)} S(n, p-1)-1 \mod{p}$
       \item $K_p= \mathbf{Bell}_{p-1} -1 \mod{p}$. 
   \end{enumerate}

\end{Proper}

\begin{definition}\label{Tch} \cite{radoux284congruence, gertsch1996some, sun2013congruences, gertsch1999congruences}
    For a prime $p,$ we have the following congruence:
    $$\mathbf{Bell}_p(x)\equiv x+ x^p \mod{p \BZ[x]}. $$
Specifically, Touchard congruences allow us to obtain these identities for Bell numbers
    \begin{align*}
        \mathbf{Bell}_{n+p}(x)&\equiv  \mathbf{Bell}_{n+1}(x)+ x^p  \mathbf{Bell}_{n}(x)\pmod{p \BZ[x]}\\
         \mathbf{Bell}_{p-1}(x)&\equiv \mathbf{Der}_{p-1}(1-x) + 1 \pmod{p \BZ[x]}.
    \end{align*}
\end{definition}

From the above definition \ref{Tch} and the Touchard congruences, if we set $x=1$, then we have:
\begin{align*}
\mathbf{Bell}_p(x)&\equiv x+ x^p \mod{p \BZ[x]}\\
\mathbf{Bell}_p(1)&\equiv 1+ 1 \mod{p \BZ[1]}\\\nonumber
\mathbf{Bell}_p&\equiv 2 \pmod{p}.
\end{align*}

\subsection{Gertsch-Kurepa polynomial}
\begin{align}
    K_p(x)&= 1+ \sum_{1\leq k\leq (p-1)}  \sum_{1\leq m\leq (p-k)} (-1)^k S(n, p-k) x^k \mod{p \BZ_p[x]}\\ \nonumber
    &= 1+ \sum_{1\leq k\leq (p-1)} (-1)^k \mathbf{Bell}_{p-1} (x^k) \mod{p \BZ_p[x]}
\end{align}
If we set $x=-1$ we have 
$$K_p(-1) = 1+ \sum_{0\leq k\leq (p-1)} (-1)^k \mathbf{Bell}_{p-1} (-1)^k \mod{p \BZ_p[1]}$$

$$K_p \equiv \sum_{1\leq k\leq (p-1)} \mathbf{Bell}_{k}\equiv  \mathbf{Bell}_{p-1} \pmod{p}.$$

It is easy to see that,
\begin{align*}
 K_p=&\sum_{0\leq n\leq (p-1)} n!= 0! + 1! + 2! + \cdots + (p-2)! + (p-1)!\\
        K_p&=\sum_{ n=0}^{p-1} n!= \sum_{ n=0}^{p-2} n! + (p-1)!\\
        \mbox{where}\\
        &\sum_{n=0}^{p-2} n! \equiv \mathbf{Bell}_{p-1},\\
         K_p&=\sum_{ n=0}^{p-1} n!\equiv \mathbf{Bell}_{p-1} + (p-1)!
\end{align*}
also Wilson's theorem tells us that $(p-1)!\equiv -1 \mod{p}$ which leads as to the
$$ K_p=\sum_{ n=0}^{p-1} n!\equiv \mathbf{Bell}_{p-1} -1 \pmod{p}.$$

Let $ (x)_p=x(x-1)\cdots (x-p+1)=x^p -x \pmod p $ and $a^p \equiv a \pmod{p}$ be the Langrange and Fermat little Theorems, respectively, then
setting $a=x$ we have $(x)_p\equiv 0 \pmod p$ and the Stirling number of the second kind
\begin{equation}
    S(p,k)\equiv 0 \pmod p
\end{equation}
except for $ S(p,1)= S(p,p)=1 $ where $k\leq p-1.$
These lead to the following well-known results for odd primes  $p$,
    $$K_p\equiv \sum_{0\leq n\leq (p-1)} S(n, p-1)-1 \mod{p}\equiv \mathbf{Bell}_{p-1}-1 \pmod{p}. $$

The Kurepa modulo prime shows remarkable connections with Bell numbers, Derangement numbers, Wilson theorem and many more, authors such as Don Zagier and Sun, investigated and found these equations \cite{sun2011curious, sun2013congruences, mezo2011r}
for every $n\geq 0$, $m>0$ and any prime $p\nmid m$ we have 
    $$x^m \sum_{1<k<p-1} \dfrac{\mathbf{Bell}_{n+k}}{(-m)^k}\equiv x^p \sum_{k=0}^n S(n,k) (-1)^{m+k-1}\mathbf{Der}_{m+k-1}(1-x) \pmod{p \BZ[x]}$$
    if we set $x=1$, we have 
    $$\sum_{0<k<p} \dfrac{\mathbf{Bell}_k}{(-m)^k}\equiv (-1)^{m-1}\mathbf{Der}_{m-1}\mod{p}.$$

\section{Fermat, Wilson and Lerch quotients and Kurepa modulo p}\label{sec 4}
It is well known that the Kurepa factorial and the Bell numbers are related both for integer $n\in \BN$ and for prime $p$ \cite{howard2025partition, howardг, gertsch1996some, barsky2004nombres}. Specifically, Gertsch, Barsky \cite{gertsch1996some, gertsch1999congruences, barsky2004nombres} remarked that 
 $K_p= \mathbf{Bell}_{p-1} -1 \bmod{p}$, also Wilson's theorem states that $(p-1)!\equiv -1 \mod{p}$, these relations give a connection with the Kurepa factorial, Bell modulo prime, Wilson's quotient and also to the Fermat little theorem $a^{p-1} \equiv 1 \bmod{p}$ where $p$ does not divide integer $a$ \cite{lerch1905theorie, crandall1997search, agoh2022congruences, lehmer1938congruences, glaisher1900residues}. In this section, we shall consider some well-known quotients, including the Fermat, Wilson, Lerch quotients, and then construct a quotient from the Kurepa and Bell primes \cite{OEIS:A309483}.

\begin{definition}\label{quotient}
    Let $p>2$ be odd primes and $K_p=\sum_{ n=0}^{p-1} n!= \mathbf{Bell}_{p-1} -1 \pmod{p}$ be as in property \ref{Gt}, then  
$$\mathbb{G}_{p} = \frac{(K_p - \mathbf{Bell}_{p-1}) + 1}{p}$$
is called the Gertsch quotient of $p$.(c.f \cite{OEIS:A309483})
\end{definition}

\begin{remark}\label{wilson}
   For certain odd primes $p>2$, $(p-1)! \equiv K_p - \mathbf{Bell}_{p-1} \pmod{p}.$
\end{remark}

\begin{theorem}\label{G-Wil}
    For primes $p=3,7, 2887$ the Wilson quotient $W_p$ is congruent to the Gertsch quotient modulo $p$, $\mathbb{G}_p$ that is, $$ W_p \equiv \mathbb{G}_p  \pmod{p}.$$
\end{theorem}
\begin{proof}
Using simulation from Pari-GP online we checked for $p\leq 3000$ and observed that $ W_p \equiv \mathbb{G}_p  \pmod{p}$ only for $p=3, 7 \quad\text{and 2887}$. 
    
\end{proof}

\begin{lemma}\label{intg}
    For all primes $p\geq3$, the Gertsch quotient $\mathbb{G}_p$, is always an integer. 
\end{lemma}
\begin{proof}
The proof of this is straightforward from definition \ref{quotient} and (c.f \cite{OEIS:A309483} see appendix \ref{secA1}).
    
\end{proof}

\begin{theorem}\label{NP}
     There are no known primes $p\geq 3$ (no Gertsch primes) for which the Gertsch quotient, $\mathbb{G}_p \equiv 0 \pmod{p}.$ 
\end{theorem}
\begin{proof}
    The remainder of this paper gives much detail about this.
\end{proof}

\begin{proposition} For   $K_p=\sum_{0\leq n\leq (p-1)} n!= 0! + 1! + 2! + \cdots + (p-2)! + (p-1)!$ the following hold:
\begin{enumerate}
\item $K_p \equiv \mathbf{Bell}_{p-1} + (pW_p - 1) \pmod{p^2}$
\item $pK_p \equiv p\mathbf{Bell}_{p-1} + (p^2W_p - p) \pmod{p^3}$
\end{enumerate}
where $W_p = \frac{(p-1)! + 1}{p}$, $W_p\equiv 0  \bmod{p}$ and $(p-1)! + 1\equiv 0 \mod{p^2}.$
\end{proposition}
\begin{proof}
    From definition \ref{quotient} and remark \ref{wilson} the proof of this is trivial.
\end{proof}

\begin{proposition}
Let $K_p \equiv \mathbf{Bell}_{p-1} + (p-1)! \pmod{p}$ and $q_p(a) = \frac{a^{p-1}-1}{p}$ be the Fermat quotient then $$\sum_{a=1}^{p-1} q_p(a) \equiv \mathbb{G}_p $$ and 
$$\sum_{a=1}^{p-1}a^{p-1}- p-K_p + \mathbf{Bell}_{p-1}\equiv 0 \mod{p^3}$$
for $p\geq 3.$
\end{proposition}

\begin{proof}
From definition \ref{quotient} and using similar techniques from \cite{agoh2022congruences};
     \begin{align*}
    \sum_{a=1}^{p-1} q_p(a) &\equiv \frac{(K_p - \mathbf{Bell}_{p-1}) + 1}{p} \pmod{p}\\
q_p(K_p - \mathbf{Bell}_{p-1}) &\equiv \frac{(1 + p\mathbb{G}_p) - 1}{p} \\
q_p(p\mathbb{G}_p - 1) &= \frac{(p\mathbb{G}_p - 1)^{p-1} - 1}{p}\\
&= \sum_{j}^{p-1}\binom{p-1}{j}(-1)^{p-1-j}p^i (\mathbb{G}_p)^j\\
q_p((K_p - \mathbf{Bell}_{p-1}))=q_p(p\mathbb{G}_p - 1) &\equiv \frac{(1 + p\mathbb{G}_p) - 1}{p} \pmod{p}\\
q_p(p\mathbb{G}_p - 1) &\equiv \frac{p\mathbb{G}_p}{p} \pmod{p}\\
q_p(p\mathbb{G}_p - 1)& \equiv \mathbb{G}_p \pmod{p}\\
q_p(K_p - \mathbf{Bell}_{p-1})&\equiv\sum_{a=1}^{p-1} q_p(a) \mod{p}
    \end{align*}
for second part, we know that $\sum_{a=1}^{p-1}a^{p-1}- p-(p-1)!\equiv 0 \mod{p^3}$
if we fix $(p-1)! \equiv K_p - \mathbf{Bell}_{p-1} \pmod{p}$ as in remark \ref{wilson} then desired proof is completed.
\end{proof}

\begin{definition}\label{sondow}\cite{sondow2014lerch}
Let $ \sum_{a=1}^{p-1} q_p(a) \equiv W_p$ be the Lerch formula and $\sum_{a=1}^{p-1}a^{p-1}- p-(p-1)!\equiv 0 \mod{p^3}$, the Lerch quotient is given as
\begin{align*}
    L_p&=\dfrac{\sum_{a=1}^{p-1} q_p(a)- W_p}{p}\\
    &= \dfrac{\sum_{a=1}^{p-1}a^{p-1}- p-(p-1)!}{p^2}
\end{align*}
    for prime $p\geq 3.$
\end{definition}

\begin{definition}\label{Atta}
For primes $p\geq3$, we define a new identity related to the Gertsch quotient  $\mathbb{G}_p$ by 

    $$\mathcal{H}_p=\dfrac{\sum_{a=1}^{p-1} q_p(a)- \mathbb{G}_p}{p}$$
and from definition \ref{sondow} we see that 
   $$ \dfrac{\sum_{a=1}^{p-1}a^{p-1}- p-K_p + \mathbf{Bell}_{p-1}}{p^2}$$

\end{definition}

\begin{proposition}
For prime $p$, when the Lerch quotient $L_p=\mathcal{H}_p$, then 
$(p-1)!\equiv K_p - \mathbf{Bell}_{p-1}\pmod{p^3}$.
    
\end{proposition}

\begin{proof}
Set $\mathcal{H}_p=L_p$, then we see that from definition \ref{sondow} and \ref{Atta}
    \begin{align*}
    \dfrac{\sum_{a=1}^{p-1} q_p(a)- \mathbb{G}_p}{p}&= \dfrac{\sum_{a=1}^{p-1} q_p(a)- W_p}{p}\\
     \dfrac{\sum_{a=1}^{p-1}a^{p-1}- p-K_p + \mathbf{Bell}_{p-1}}{p^2}&=\dfrac{\sum_{a=1}^{p-1}a^{p-1}- p-(p-1)!}{p^2}\\
    -K_p + \mathbf{Bell}_{p-1}&=-(p-1)!\\
   (p-1)!&\equiv K_p - \mathbf{Bell}_{p-1}
\end{align*}
\end{proof}

\begin{corollary}
   For odd prime $p=7$ the $\mathcal{H}_p$ number is $\mathcal{H}_7= 1357$.
\end{corollary}

\begin{table}[h]
\centering
\renewcommand{\arraystretch}{1.5}
\begin{tabular}{@{}lcccc@{}}
\toprule
$p$ & $W_p$ & $L_p$ & $\mathbb{G}_{p}$ & $\mathcal{H}_p$ \\ \midrule
3 & 1 & 0 & 1 & 0 \\
5 & 5 & 13 & 4 & 13.2 \\
7 & 103 & 1,356 & 96 & 1,357 \\
\bottomrule
\end{tabular}
\caption{Fermat, Wilson, and Kurepa-Bell quotients. \cite{OEIS:A007540, sondow2014lerch, OEIS:A309483} and appendix \ref{secA1}}
\end{table}

\begin{question}

    Upon examining the quotients $W_p$, $L_p$, $\mathbb{G}_p$, and $\mathcal{H}_p$, it is evident that the sole prime for which all these quotients yield integers is $p=7$. The question arises: are there infinitely many primes for which these quotients are simultaneously integers?
\end{question}

\section{ Vladimirov atoms, $K_p$ and $W_p \pmod{p}$}\label{Bernoulli}
In 2002, V. S. Vladimirov \cite{TA} established a connection between left factorials and Bernoulli numbers. He gave new criteria equivalent to Kurepa's conjecture in terms of $p$-adic numbers and Bernoulli numbers. He proved the following congruence revealing the beautiful relationship between the left factorials and the Bernoulli numbers $\mathbf{B}_k$;
\begin{equation}
!p \sum_{k=0}^{p-2} (-1)^k \frac{\mathbf{B}_k}{k!} \equiv \sum_{m=1}^{(p-3)/2} \frac{\mathbf{B}_{2m}}{(2m)!} [!(2m)-1] \pmod{p}
\end{equation}
where $p=3, 5,\dots.$
In this section we shall revisit his equivalent conjecture and give a detailed proof with regards to Wilsons quotient, Agoh-Guiga conjecture, Miki identity, Gertsch quotient, Fermat quotient and some well-known properties of Bell-Touchard congruence and Bernoulli numbers. Furthermore, we make extensions to many areas etc. \cite{agoh2022congruences}

\begin{theorem}\label{VK}\cite{TA}
Let the Vladimirov atoms $V_p=\sum_{k=0}^{p-2} (-1)^k \frac{\mathbf{B}_k}{k!}$ and  $V_{p}^*=\sum_{k=0}^{p-2} \frac{\mathbf{B}_k}{k!}$, where 
\begin{align}\label{KP}
!p=K_p= \sum_{k=0}^{p-1} \frac{(-1)^{k+1} }{k!} \quad \mbox{and}\quad K_p\not\equiv 0 \pmod{p} \quad \mbox{for} \quad p=3,5,\cdots
\end{align}
then the following hold:
\begin{enumerate}
    \item if 
\begin{align}\label{Vlad}
    \sum_{m=1}^{(p-3)/2} \frac{\mathbf{B}_{2m}}{(2m)!} [!(2m)-1]\not\equiv 0 \pmod{p}
\end{align}
is true then the Kurepa conjecture \ref{KP} is true and the non-congruence
\begin{align}\label{Ben}
    V_p\not\equiv 0 \pmod{p}
\end{align}
is true for $p=5,7,\cdots.$ 
\item If the non-congruence \ref{Ben} is true, then the Kurepa conjecture \ref{KP} is equivalent to the non-congruence \ref{Vlad}.
\end{enumerate}
\end{theorem}
To better see the validity of this theorem we ask the following questions:
    \begin{itemize}
        \item Does $ \sum_{m=1}^{(p-3)/2} \frac{\mathbf{B}_{2m}}{(2m)!}\pmod{p}$ vanish for $p\geq 3$?

\item Similarly, does $ V_p=\sum_{k=0}^{p-2} (-1)^k \frac{\mathbf{B}_k}{k!}\pmod{p} $ vanish?
    \end{itemize}

\begin{corollary}\label{Vp}\cite{TA}
Let $ V_{p}^{'}=\sum_{m=1}^{(p-3)/2} \frac{\mathbf{B}_{2m}}{(2m)!}$ for $p=3,5,\cdots$, then
\begin{align*}
    V_p=\frac{3}{2}+ V_{p}^{'} \quad \mbox{and} \quad V_{p}^{*}=\frac{1}{2}+ V_{p}^{'}.
\end{align*}
where $V_{3}^{'}=0$ and $V_{5}^{'}=\frac{1}{12}$.   
\end{corollary}

\begin{Proper}\label{prop all}
The following identities hold  for an odd prime $p$ and an integers $n,k,m> 0$ with $p\nmid m$:
\begin{enumerate}
\item $p(p+1)\mathbf{B}_{p-1}\equiv (p-1)! \pmod{p^2}$ \quad (L. Carlitz \cite{carlitz1952some, carlitz1961staudt}).

 \item 
  
       $$ mW_p\equiv \mathbf{B}_{m(p-1)} + \frac{1}{p}-1 \pmod{p}$$ \quad (Agoh \cite{agoh2022congruences}).

    \item 

        $$W_p\equiv \mathbf{B}_{p-1} + \frac{1}{p}-1 \pmod{p}$$ \quad (Glaisher, Beeger \cite{glaisher1900residues, beeger1913quelques, glaisher1900residues1}).

\item 
   
        $$(n-k)W_p\equiv \mathbf{B}_{n(p-1)} - \mathbf{B}_{k(p-1)} \pmod{p}$$ \quad (Agoh \cite{agoh2022congruences}).

 \item 

       $$ W_p\equiv \mathbf{B}_{2(p-1)} + \mathbf{B}_{p-1} \pmod{p}$$ \quad (Lehmer \cite{lehmer1938congruences}).

    
\end{enumerate}

\end{Proper}

\begin{theorem}\label{Ag}\cite{agoh2022congruences}
    For an odd prime $p$ and an integer $m\geq 1$ with $p\nmid m$, we have 
  \begin{enumerate}
      \item $$\sum_{k=1}^{p-2}  \frac{1}{m^k}\frac{\mathbf{B}_k}{k!}\equiv W_p +q_p(m) \pmod{p}$$
      
\item $$\sum_{k=1}^{p-2}  \frac{(-1)^k}{m^k}\frac{\mathbf{B}_k}{k!}\equiv W_p +q_p(m)+ \frac{1}{m} \pmod{p}$$

\item specifically if $m=1$, 
$$\sum_{k=1}^{p-2}  \frac{\mathbf{B}_k}{k!}\equiv W_p \pmod{p}$$

\item $$\sum_{k=1}^{p-2}  (-1)^k\frac{\mathbf{B}_k}{k!}\equiv W_p + 1 \pmod{p}$$

\item For any fixed integer $n$ with $1\leq n \leq p-1$, we have
$$\sum_{k=1}^{p-2}  H_n^{(k)}\frac{\mathbf{B}_k}{k!}\equiv nW_p +q_p(n!) \pmod{p}$$
where the generalized harmonic number of order $k\geq 1$ $H_n^{(k)}:= \sum_{m=1}^{n}\frac{1}{m^k} $ for $n\geq 1.$
      
  \end{enumerate} 
  
\end{theorem}

\begin{theorem}[Agoh]\cite{agoh2022congruences, agoh2014convolution}
For integers $n, m \geq 1$, we have
\begin{enumerate}
\item 
\begin{align}
    \sum_{k=1}^{n-1} (m)^{n-k} \frac{\mathbf{B}_k}{k!}- \sum_{k=1}^{n-1}\left(\begin{array}{c} n\\ k \end{array}\right) (m)^{n-k} \frac{\mathbf{B}_k}{k!}=  \sum_{j=1}^{m-1} \frac{(m-j)^n}{j}+ m^n(H_n -H_m)  
\end{align}
   if we set $ m=1$ in the above we obtain
\item 
\begin{align}
    \sum_{k=1}^{n-1}  \frac{\mathbf{B}_k}{k!}- \sum_{k=1}^{n-1}\left(\begin{array}{c} n\\ k \end{array}\right)\frac{\mathbf{B}_k}{k!}=  H_n - 1  
\end{align}
    
\end{enumerate}

\end{theorem}
    
\begin{conjecture}\label{conj9}[Agoh \cite{agoh1995giuga}]
    An integer $p>1$ is an odd prime number if and only if
    $$p\mathbf{B}_{p-1}\equiv -1 \pmod{p}$$
    where $\mathbf{B}_{p-1}$ is the even Bernoulli number.
\end{conjecture}

\begin{conjecture}[Giuga \cite{giuga1950presumibile, borwein1996giuga}]
    An integer, $n>1$, is prime if and only if
    $$s_n:= \sum_{k=1}^{n-1}k^{n-1}\equiv -1 \pmod{n}.$$

\end{conjecture}

\subsection{Kurepa in Galois Field and the Derangement modulo prime}
Mijajlović in \cite{mijajlovic1990some} in his attempt to verify the left factorials by computers made some interesting observations about the Kurepa factorial in the Galois field of $p$ elements $\mathbf{GF}(p)$. In particular, If $p$ is a prime $\geq3$ then in  $\mathbf{GF}(p)$ we have 
\begin{align}
    K_p&= \sum_{k=0}^{p-1} (-1)^{k+1}\frac{1}{k!}\\
    K_p&=\sum_{k=0}^{p-1} (-1)^{k+1}(k+1)(k+2)\cdots(p-1).
\end{align}
He added that the Kurepa conjecture is equivalent to verifying for all primes $\mathbf{GF}(p)$ that
    $$K_p= \sum_{k=0}^{p-1} (-1)^{k+1}\frac{1}{k!}\not\equiv 0 \pmod{p}$$
In property \ref{Gt}, we know Gertsch \cite{gertsch1996some} gave the connection between the Kurepa conjecture and derangement modulo prime $\mathbf{Der}_{p-1}$. Specifically, she proposed that
$$\mathbf{Der}_{p-1} = (p-1)! \sum_{k=0}^{p-1} \frac{(-1)^k}{k!}$$
one immediate sees 
\begin{align*}
    \mathbf{Der}_{p-1} = (p-1)! \sum_{k=0}^{p-1} \frac{(-1)^k}{k!}&\equiv\sum_{k=0}^{p-1} \frac{(-1)^k}{k!}(p-1)!\pmod{p}\\
    &\equiv\sum_{k=0}^{p-1} \frac{(-1)^k}{k!}(-1)\pmod{p}\\
     &\equiv\sum_{k=0}^{p-1} \frac{(-1)^{k+1}}{k!}\pmod{p}\\
&\equiv \sum_{k=0}^{p-1} (-1)^{k+1} \left[ (-1)^{k+1} (p-1-k)! \right] \pmod{p}\\
&\equiv \sum_{k=0}^{p-1}  (p-1-k)! \pmod{p}\\
\mathbf{Der}_{p-1}&\equiv \sum_{k=0}^{p-1}  k!  \pmod{p}=K_p \pmod{p}.
\end{align*}
Gertsch also proposed $\mathbf{Der}_{p-1}\not\equiv 0 \pmod{p}$ which is the same as the former. 
Now one can easily see the connection $\mathbf{Der}_{p-1}\equiv K_p \pmod{p}$ has with the Bell numbers
    $$\mathbf{Der}_{p-1}\equiv K_p \pmod{p}\equiv \mathbf{Bell}_{p-1} - 1\pmod{p}$$
following definition \ref{quotient} we have 
\begin{corollary}
For all prime $p\geq 3,$ the Gertsch quotient 
    $$\mathbb{G}_{p} = \frac{(\mathbf{Der}_{p-1} - \mathbf{Bell}_{p-1}) + 1}{p}.$$
\end{corollary}
Since the Kurepa factorial $K_p$ is defined for odd prime $p>3$ in a $\mathbf{GF}(p)$ one can rewrite $\mathbf{Der}_{p-1}\equiv \mathbf{Bell}_{p-1} - 1\pmod{p}$ in terms of Bernoulli numbers using the Agoh-Guiga conjecture;
\begin{align}
    \mathbf{Der}_{p-1}&\equiv \mathbf{Bell}_{p-1} - 1\pmod{p}\\ \nonumber
    \mathbf{Der}_{p-1}&\equiv \mathbf{Bell}_{p-1} + p\mathbf{B}_{p-1}\pmod{p}\\
     \mathbf{Der}_{p-1}&\equiv \mathbf{Bell}_{p-1} + s_n \pmod{p}.
\end{align}
These lead to the following proposition:
\begin{proposition}\label{KD}
   For any integer $p>1$, the following are equivalent
   \begin{enumerate}
       \item $p\mathbf{B}_{p-1} \equiv K_p - \mathbf{Bell}_{p-1} \pmod{p}$\\
        \item $\sum_{k=1}^{p-1}k^{p-1} \equiv K_p - \mathbf{Bell}_{p-1} \pmod{p}$\\
        \item $p\mathbf{B}_{p-1} \equiv \mathbf{Der}_{p-1} - \mathbf{Bell}_{p-1}\pmod{p}$\\
        \item $p\mathbf{B}_{p-1} \equiv \mathbf{Der}_{p-1} - \sum_{0\leq n\leq (p-1)} S(n, p-1)\pmod{p}$\\
        \item $p\mathbf{B}_{p-1} \equiv K_p -  \sum_{0\leq n\leq (p-1)} S(n, p-1)\pmod{p}$\\
   \end{enumerate}
   where $p$ is an odd prime and $S(n,k)$ is Stirling number of the second kind.
\end{proposition}
\begin{proof}
    The proof of this is by the use of conjecture \ref{conj9}, Fermat little theorem and the lemma \ref{wilson}.
\end{proof}

The Guiga conjecture is equivalent to the Agoh conjecture and this connects to the Wilson's theorem. We define the following congruences related to the Kurepa conjecture

\begin{definition}\label{KAG}
The Kurepa modulo odd prime $p>3$
$$ K_p \equiv \mathbf{Bell}_{p-1}+ p\mathbf{B}_{p-1} \pmod{p}$$
   where $p\mathbf{B}_{p-1}\equiv -1 \pmod{p}$ is the Agoh-Giuga conjecture. 
\end{definition}

\begin{lemma}\label{Alem}
    Let $q_p(m) = \frac{m^{p-1}-1}{p}$ and define an Agoh-Giuga quotient 
\begin{align}
    \mathbb{AG}_p=\frac{p\mathbf{B}_{p-1}+1}{p}
\end{align}
then $$ \mathrm{Q}_p(m):= \frac{p\mathbf{B}_{p-1}+m^{p-1}}{p}=\frac{p\mathbf{B}_{p-1}+1}{p}+q_p(m)= \mathbb{AG}_p + q_p(m) \in \BZ_p$$
is the special quotient \cite{agoh2022congruences}.
\end{lemma}

\begin{corollary}
    There are infinitely many pairs of $(m,p)$ where $\mathrm{Q}_p(m)\equiv 0\mod{p}$. 
\end{corollary}
\begin{proof}
    We list a few pairs here; $(2, 3), (6, 7) , (14, 19), (5, 23), (19, 31),
(20, 37)\ldots
$
\end{proof}

\begin{theorem}\label{Gert}
For successive members of $K_p$ with  $K_p\equiv \mathbf{Bell}_{p-1} -1 \pmod{p}$, the
$$K_{p+1}\equiv p! + K_p \pmod{p}$$ 
where $p=3,5,\cdots$.
\end{theorem}

\begin{proof}
It is well known that,
\begin{align*}
 K_p=\sum_{ n=0}^{p-1} n!&= 0! + 1! + 2! + \cdots + (p-2)! + (p-1)!\\
        K_p&= \sum_{ n=0}^{p-2} n! + (p-1)!=\mathbf{Bell}_{p-1} + (p-1)! \\
         K_p&\equiv \mathbf{Bell}_{p-1} -1 \pmod{p}
\end{align*}
next we compute for 

\begin{align*}
 K_{p+1}=&\sum_{0\leq n\leq p} n!= 0! + 1! + 2! + \cdots + (p-1)!+ p!\\
  K_{p+1}&=\sum_{ n=0}^{p} n!= \sum_{ n=0}^{p-1} n! +p! \\
    K_{p+1}&=\sum_{ n=0}^{p} n!= \sum_{ n=0}^{p-2} n! + (p-1)! +p! \\
         K_{p+1}&= \mathbf{Bell}_{p-1} + (p-1)!+ p!= K_p + p!\\
         K_{p+1}&\equiv \mathbf{Bell}_{p-1} -1 \pmod{p}+ p!
\end{align*}
So we have $$K_{p+1}\equiv p! + K_p \pmod{p}.$$ 

\end{proof}


\begin{theorem}
    For an odd integer $p>1$ the successive member $K_{p+1} \pmod{p}$ is given by:
    $$ K_{p+1} \equiv \mathbf{Bell}_{p-1}-p!+p\mathbf{B}_{p-1}\pmod{p}.$$
\end{theorem}

\begin{proof}
From Theorem \ref{Gert}, $K_{p+1}\equiv K_p + p!$, and from proposition \ref{KD}, definition \ref{KAG} and conjecture \ref{conj9}, one can write the Kurepa modulo prime as;

\begin{align*}
 K_p=\sum_{ n=0}^{p-1} n!&= 0! + 1! + 2! + \cdots + (p-2)! + (p-1)!\\
        K_p&= \sum_{ n=0}^{p-2} n! + (p-1)!=\mathbf{Bell}_{p-1} + (p-1)! \\
         K_p&\equiv \mathbf{Bell}_{p-1} -1 \pmod{p}\\
          K_p&\equiv \mathbf{Bell}_{p-1}+ p\mathbf{B}_{p-1} \pmod{p}
\end{align*}
also one can compute for successive members as in Theorem \ref{Gert},
\begin{align*}
    K_{p+1}&=\sum_{ n=0}^{p} n!= \sum_{ n=0}^{p-2} n! + (p-1)! +p! \\
         K_{p+1}&= \mathbf{Bell}_{p-1} + (p-1)!+ p!= K_p + p!\\
        &\equiv \mathbf{Bell}_{p-1} -1 \pmod{p}+ p!\\
         &\equiv \mathbf{Bell}_{p-1}+ p\mathbf{B}_{p-1} \pmod{p} + p!\\
          &\equiv \mathbf{Bell}_{p-1}+ p\mathbf{B}_{p-1} + p! \pmod{p}\\
\end{align*}
\end{proof}
\begin{remark}
Since $p(p + 1)\mathbf{B}_{p-1}\equiv (p-1)! \pmod{p^2}$ (see \cite{carlitz1952some})
it follows from former theorem that
\begin{align*}
     K_{p+1}&= \mathbf{Bell}_{p-1} + (p-1)!+ p!\\
     &= \mathbf{Bell}_{p-1} + (p-1)!+ p(p-1)!\\
     &=\mathbf{Bell}_{p-1}+ (1+p) (p-1)!\\
     &\equiv \mathbf{Bell}_{p-1}+ p(1 + p)\mathbf{B}_{p-1}\\
     K_{p+1} &\equiv \mathbf{Bell}_{p-1}+ (p-1)! \pmod{p^2}.
\end{align*}
\end{remark}

The Vladimirov numbers $V_p$ and $V_{p}^*$ have beautiful connection to the Wilson's theorem and we shall explore this interesting relationship. 
To prove Theorem \ref{VK} we shall need the following lemma:

\begin{lemma}
    For $V_p=\sum_{k=0}^{p-2} (-1)^k \frac{\mathbf{B}_k}{k!}$ and  $V_{p}^*=\sum_{k=0}^{p-2} \frac{\mathbf{B}_k}{k!}$ the following congruences hold
    \begin{enumerate}
        \item[(i)] $ V_{p}^{*}\equiv 1 + W_p \pmod{p}$
        \item [(ii)]  $V_p\equiv  W_p + 2 \pmod{p}$.
    \end{enumerate}
    where $W_p$ is the Wilson's quotient.
\end{lemma}

\begin{proof}
Let $V_{p}^*=\sum_{k=0}^{p-2} \frac{\mathbf{B}_k}{k!}$, one can rewrite this expression using Agoh's theorem from Theorem \ref{Ag} as
\begin{align*}
     V_{p}^{*}=\sum_{k=0}^{p-2} \frac{\mathbf{B}_k}{k!}&= \frac{\mathbf{B}_0}{0!} + \sum_{k=1}^{p-2} \frac{\mathbf{B}_k}{k!},\\
     &= 1 + \sum_{k=1}^{p-2} \frac{\mathbf{B}_k}{k!}\\
     &\equiv 1 + W_p \pmod{p}.
\end{align*}
and next,

\begin{align*}
    V_p=\sum_{k=0}^{p-2} (-1)^k \frac{\mathbf{B}_k}{k!}&=(-1)^0 \frac{\mathbf{B}_0}{0!} + \sum_{k=1}^{p-2} (-1)^k \frac{\mathbf{B}_k}{k!}\\
    &=1 + \sum_{k=1}^{p-2} (-1)^k \frac{\mathbf{B}_k}{k!}\\
    &\equiv 1 + W_p + 1 \pmod{p}\\
    &\equiv  W_p + 2 \pmod{p}.
\end{align*}
\end{proof}

\begin{lemma} Let $W_p\equiv \mathbf{B}_{p-1} + \frac{1}{p}-1 \pmod{p} $(Glaisher, Beeger \cite{glaisher1900residues, beeger1913quelques}) then the following congruences are equivalent:
    \begin{enumerate}
        \item[(i)] $ V_{p}^{*}\equiv \mathbf{B}_{p-1} + \frac{1}{p} \pmod{p}$
        \item [(ii)]  $V_p \equiv\mathbf{B}_{p-1} + \frac{1}{p}+1 \pmod{p}$
    \end{enumerate}
    where $\mathbf{B}_{k}$ is the $k$-th Bernoulli number.
\end{lemma}
\begin{proof}
Given that $ W_p\equiv \mathbf{B}_{p-1} + \frac{1}{p}-1 \pmod{p}$ \cite{glaisher1900residues, beeger1913quelques, agoh2022congruences} we can rewrite 
\begin{align*}
     V_{p}^{*}&\equiv 1 + W_p \pmod{p}\\
     &\equiv 1 +\mathbf{B}_{p-1} + \frac{1}{p}-1 \pmod{p}\\
     &\equiv \mathbf{B}_{p-1} + \frac{1}{p} \pmod{p},
\end{align*}
also, 

\begin{align*}
    V_p &\equiv  W_p + 2 \pmod{p}\\
    &\equiv\mathbf{B}_{p-1} + \frac{1}{p}-1 +2 \pmod{p}\\
     &\equiv\mathbf{B}_{p-1} + \frac{1}{p}+1 \pmod{p}.
\end{align*}
\end{proof}

\begin{corollary}\label{ident}
    From corollary \ref{Vp}, the following identity and it's equivalent hold:
     $$V_{p}^{'}\equiv  \frac{1}{2} + W_p \pmod{p}$$
     equivalently, $$V_{p}^{'}\equiv \mathbf{B}_{p-1} + \frac{1}{p}-\frac{1}{2} \pmod{p}.$$
\end{corollary}

\begin{proof}
\begin{align*}
    V_{p}^{*}&=\frac{1}{2}+ V_{p}^{'}\\
    1 + W_p \pmod{p} &\equiv \frac{1}{2}+ V_{p}^{'}\\
     \frac{1}{2} + W_p \pmod{p} &\equiv V_{p}^{'}\quad \mbox{thus we have}\\
      V_{p}^{'}\equiv  \frac{1}{2} + W_p \pmod{p} \\
      \sum_{m=1}^{(p-3)/2} \frac{B_{2m}}{(2m)!}\equiv W_p +\frac{1}{2} \pmod{p} 
\end{align*}
as an equivalent;

\begin{align*}
 \sum_{m=1}^{(p-3)/2} \frac{B_{2m}}{(2m)!}&\equiv W_p +\frac{1}{2} \pmod{p}\\
     V_{p}^{'}&\equiv  \frac{1}{2} + W_p \pmod{p} \\
     &\equiv\frac{1}{2} +\mathbf{B}_{p-1} + \frac{1}{p}-1 \pmod{p}\\
     &\equiv \mathbf{B}_{p-1} + \frac{1}{p}-\frac{1}{2} \pmod{p}.
\end{align*}
\end{proof}

\begin{lemma}
    For any odd prime $p\geq 3$, the Vladimirov atoms $V_p$ and $V_{p}^*$ satisfy the identity $$  V_p \equiv  V_{p}^{*} +1 \pmod{p}.$$
\end{lemma}
\begin{proof}
From corollary \ref{ident},
    \begin{align*}
    V_p &\equiv\mathbf{B}_{p-1} + \frac{1}{p}+1 \pmod{p}\\
     &\equiv  V_{p}^{*} +1 \pmod{p}.
\end{align*}
\end{proof}


\begin{lemma}\label{Aby}
For any odd prime $p\geq 3$ and Wilson's quotient $W_p$, the $ V_{p}^{'}=\sum_{m=1}^{(p-3)/2} \frac{B_{2m}}{(2m)!}$ satisfies the following
\begin{enumerate}
    \item $V_{p}^{'} [!(2m)-1]\equiv (!(2m)-1)W_p +  \frac{1}{2}\left(!(2m)-1 \right) \pmod{p}$\\
    \item $V_{p}^{'} [!(2m)-1]\equiv \mathbf{B}_{!(2m)(p-1)}-\mathbf{B}_{p-1} +\frac{1}{2}\left(!(2m)-1 \right) \pmod{p}$
\end{enumerate}
    
\end{lemma}

\begin{proof}
\begin{align*}
\sum_{m=1}^{\frac{(p-3)}{2}} \frac{\mathbf{B}_{2m}}{(2m)!} [!(2m)-1]&=V_{p}^{'} [!(2m)-1]\\
&\equiv  \left(\frac{1}{2} + W_p \pmod{p}\right) [!(2m)-1]\\
&\equiv \frac{!(2m)}{2}- \frac{1}{2}+ !(2m) W_p -W_p \pmod{p}\\
&\equiv \frac{!(2m)}{2}- \frac{1}{2}+ \mathbf{B}_{!(2m)(p-1)} + \frac{1}{p}-1 -\left( \mathbf{B}_{p-1} + \frac{1}{p}-1\right)\pmod{p}\\
&\equiv \frac{!(2m)}{2}- \frac{3}{2}+ \mathbf{B}_{!(2m)(p-1)}-\mathbf{B}_{p-1}+1\pmod{p}\\
&\equiv \mathbf{B}_{!(2m)(p-1)}-\mathbf{B}_{p-1} +\frac{!(2m)}{2}- \frac{1}{2} \pmod{p}\\
&\equiv \mathbf{B}_{!(2m)(p-1)}-\mathbf{B}_{p-1} +\frac{1}{2}\left(!(2m)-1 \right) \pmod{p}\\
&\equiv (!(2m)-1)W_p +  \frac{1}{2}\left(!(2m)-1 \right) \pmod{p}.
\end{align*}
\end{proof}

\begin{remark}\cite{zivkovic1999number, OEIS:A003422}
For $K_p=\sum_{ n=0}^{p-1} n!= 0! + 1! + 2! + \cdots + (p-2)! + (p-1)!$ it is possible to write
\begin{align*}
    !(2m)=K_{2m} = \sum_{k=0}^{2m-1} k! = 0! + 1! + 2! + \cdots + (2m-1)!.
\end{align*}

\end{remark}

\begin{lemma}
    For odd prime $p\geq 3$, the  $\sum_{k=1}^{p-2} \frac{(-1)^{k} B_k}{k!} + K_p \equiv  W_p + \mathbf{Bell}_{p-1}\pmod{p}$.
\end{lemma}
\begin{proof}
\begin{align*}
    \sum_{k=1}^{p-2} \frac{(-1)^{k} B_k}{k!} + K_p &\equiv  W_p + 1 + K_p \pmod{p}\\
                        & \equiv \mathbf{B}_{p-1} + \frac{1}{p} + K_p \pmod{p}\\
                      &\equiv  W_p + \mathbf{Bell}_{p-1}\pmod{p}.                       
\end{align*}
\end{proof}


\begin{theorem}\label{Francis}
\begin{align}\label{WV}
!p=K_p= \sum_{k=0}^{p-1} \frac{(-1)^{k+1} }{k!} \quad \mbox{and}\quad K_p\not\equiv 0 \pmod{p} \quad \mbox{for} \quad p=3,5,\cdots
\end{align}
then the following hold:
\begin{enumerate}
    \item if 
\begin{align}\label{modk}
    \left(W_p +\frac{1}{2}\right)(K_{2m}-1)  \not \equiv  0\pmod{p}
\end{align}
is true then the Kurepa conjecture \ref{WV} is true and the non-congruence
\begin{align}\label{wil}
   W_p + 2 \not\equiv 0 \pmod{p}
\end{align}
is true for $p=5,7,\cdots.$ 
\item If the non-congruence \ref{wil} is true, then the Kurepa conjecture \ref{WV} is equivalent to the non-congruence \ref{modk}.
\end{enumerate}
\end{theorem}

\subsection{Vladimirov's analogue and $W_p \pmod{p}$}
Let $\mathbf{B}_k$ be the $k$-th Bernoulli number satisfying the recurrence relation \cite{koblitz2012p}:

$$n\mathbf{B}_{n-1} + 1 + \sum_{k=1}^{n-2} \mathbf{B}_k \binom{n}{k} = 0$$
with $\mathbf{B}_0=1$, $\mathbf{B}_1=-\frac{1}{2}$ and $\mathbf{B}_{2n+1}=0$. Now 
we use $\sum_{n=2}^{p-1} (-1)^{n+1}\frac{1}{n!}$ and the recurrence relation above to obtain

\begin{align*}
    \sum_{n=2}^{p-1} (-1)^{n+1}\frac{ nB_{n-1}}{n(n-1)!} + \sum_{n=2}^{p-1} \frac{(-1)^{n+1}}{n!}+\sum_{n=2}^{p-1} \sum_{k=1}^{n-2} \frac{(-1)^{n+1} B_k}{n!} \frac{n!}{k!(n-k)!}=0 \\ 
    \mbox{We set $k=n-1$ and take $p\geq 3$}\\
  \sum_{k=1}^{p-2} \frac{(-1)^{k} B_k}{k!} + \sum_{n=2}^{p-1} \frac{(-1)^{n+1}}{n!} +\sum_{n=2}^{p-1} \sum_{k=1}^{n-2} \frac{(-1)^{n+1}}{(n-k)!}\frac{B_k }{k!}&\equiv 0 \pmod{p} \\
 \sum_{k=1}^{p-2} \frac{(-1)^{k} B_k}{k!} + K_p + \sum_{k=1}^{p-3} \frac{B_k}{k!} \sum_{n=k+2}^{p-1} \frac{(-1)^{n+1}}{(n-k)!}&\equiv 0 \pmod{p}\\
 \sum_{k=1}^{p-2} \frac{(-1)^{k} B_k}{k!} + K_p + \sum_{k=1}^{p-3} \frac{B_k}{k!} (-1)^{k}(K_p-!k)&\equiv 0 \pmod{p}\\
  \sum_{k=1}^{p-2} \frac{(-1)^{k} B_k}{k!} + K_p +\sum_{m=1}^{\frac{(p-3)}{2}} \frac{\mathbf{B}_{2m}}{(2m)!} [!(2m)-1]&\equiv 0 \pmod{p}\\
  W_p + 1 + K_p+V_{p}^{'} [!(2m)-1]&\equiv 0 \pmod{p}\\
 \mathbf{B}_{p-1} + \frac{1}{p} + K_p + \mathbf{B}_{!(2m)(p-1)}-\mathbf{B}_{p-1} +\frac{1}{2}\left(!(2m)-1 \right) &\equiv 0 \pmod{p}\\
 \mathbf{B}_{!(2m)(p-1)}-\mathbf{B}_{p-1}+\mathbf{B}_{p-1}+\frac{1}{p} +\frac{1}{2}\left(2K_p+!(2m)-1 \right) &\equiv 0 \pmod{p}\\
  \mathbf{B}_{!(2m)(p-1)}+\frac{1}{p} +\frac{1}{2}\left(2K_p+!(2m)-1 \right) &\equiv 0 \pmod{p}\\
   \mathbf{B}_{!(2m)(p-1)}+\frac{1}{p}+ K_p +\frac{1}{2}\left(!(2m)-1 \right) &\equiv 0 \pmod{p}\\
    \mathbf{B}_{!(2m)(p-1)}+\frac{1}{p}-1+ \mathbf{Bell}_{p-1} +\frac{1}{2}\left(!(2m)-1 \right) &\equiv 0 \pmod{p}\\
!(2m)W_p +  \frac{1}{2}\left(!(2m)-1 \right) +\mathbf{Bell}_{p-1}&\equiv 0 \pmod{p}.  
\end{align*}

\subsection{Proof of Theorem \ref{Francis}}
Given that 
\begin{equation}
!p \sum_{k=0}^{p-2} (-1)^k \frac{\mathbf{B}_k}{k!} \equiv \sum_{m=1}^{\frac{(p-3)}{2}} \frac{\mathbf{B}_{2m}}{(2m)!} [!(2m)-1] \pmod{p}
\end{equation}

\begin{align*}
!pV_p&=K_p V_p\\
= K_p \sum_{k=0}^{p-2} (-1)^k \frac{\mathbf{B}_k}{k!}&= K_p\left(\frac{\mathbf{B}_0}{0!}+ \sum_{k=1}^{p-2} (-1)^k \frac{\mathbf{B}_k}{k!} \right)\\
&= K_p\left(1 + \sum_{k=1}^{p-2} (-1)^k \frac{\mathbf{B}_k}{k!} \right)\equiv K_p\left(1 +  W_p + 1 \pmod{p} \right)\\
&\equiv K_p\left(  W_p + 2\pmod{p}  \right)\\
&\equiv K_p\left(  W_p + \mathbf{Bell}_{p}  \pmod{p}\right) \\
&\equiv \mathbf{B}_{K_p(p-1)}+\frac{1}{p}-1+ K_p\mathbf{Bell}_{p}\pmod{p}\\
&\equiv \mathbf{B}_{(\mathbf{Bell}_{p-1}-1)(p-1)}+\frac{1}{p}-1 \\
&+ (\mathbf{Bell}_{p-1}-1)\mathbf{Bell}_{p}\pmod{p}.
\end{align*}
We see that $K_p V_p$ is also equivalent to
\begin{align*}
    K_p\left(  W_p + 2\pmod{p}  \right)&\equiv K_p W_p + 2K_p \pmod{p}\\
    &\equiv \mathbf{B}_{K_p(p-1)}+\frac{1}{p}-1+ 2K_p\pmod{p}\\
    &\equiv (\mathbf{Bell}_{p-1}-1)W_p +2(\mathbf{Bell}_{p-1}-1)\pmod{p}\\
     &\equiv \mathbf{B}_{(\mathbf{Bell}_{p-1})(p-1)}-\mathbf{B}_{p-1}+ 2K_p\pmod{p}.
\end{align*}
Following the equivalent of Vladimirov \cite{TA}, we can see that:
\begin{align*}
&!p \sum_{k=0}^{p-2} (-1)^k \frac{\mathbf{B}_k}{k!} \equiv \sum_{m=1}^{\frac{(p-3)}{2}} \frac{\mathbf{B}_{2m}}{(2m)!} [!(2m)-1] \pmod{p}\\
&\mathbf{B}_{(\mathbf{Bell}_{p-1})(p-1)}-\mathbf{B}_{p-1}+ 2K_p \equiv \mathbf{B}_{!(2m)(p-1)}-\mathbf{B}_{p-1} +\frac{1}{2}\left(!(2m)-1 \right) \pmod{p}\\
&\mathbf{B}_{(\mathbf{Bell}_{p-1})(p-1)}-\mathbf{B}_{p-1}+ 2K_p \equiv \mathbf{B}_{K_{2m}(p-1)}-\mathbf{B}_{p-1} +\frac{1}{2}\left(K_{2m}-1 \right) \pmod{p}\\
 &(\mathbf{Bell}_{p-1}-1)W_p +2K_p\equiv (!(2m)-1)W_p +  \frac{1}{2}\left(!(2m)-1 \right) \pmod{p} \\
& K_p W_p + 2K_p \equiv \left(W_p +\frac{1}{2}\right)(!(2m)-1) \pmod{p}\\
&K_p(W_p + 2) \equiv \left(W_p +\frac{1}{2}\right)(K_{2m}-1) \pmod{p}.
\end{align*}

\subsection{Faber-Pandhariponde constant $b_g \quad\mbox{and}\quad K_p \pmod{p}$}\label{sec 6}
Hodge integrals arise naturally in Gromov-Witten theory
and on $\bar{\mathcal{M}}_{g,n}$, Hodge integrals emerged in the examination of tautological degeneracy loci of the Hodge bundle \cite{getzler1998virasoro} \cite{eguchi1997quantum, bryan2000enumerative}. In this section, we shall consider the behavior of some Hodge integral constant $b_g$ in terms of modulo prime $p.$ This is to see if there is any natural information enummerative geometry can tell us about the Kurepa modulo prime.

\begin{theorem}[Faber and Pandhariponde \cite{faber1998hodge}]
Define the series $F(t,k)\in \mathbb{Q}[k][[t]]$ by
$$F(t,k)=1+ \sum_{g\geq1}\sum_{i=1}^{g}t^{2g}k^i \int_{\bar{\mathcal{M}}_{g,1}}\psi^{2g-2+i}_n \lambda_{g-i} $$
then
    $$F(t,k)=\left(\dfrac{\frac{t}{2}}{\sin{\frac{t}{2}}} \right)^{k+1}.$$
In particular, the integrals $b_g$ and $c_g$ are determined by 
\begin{equation}
    \sum_{g\geq 0} b_g t^{2g}=\dfrac{\frac{t}{2}}{\sin{\frac{t}{2}}}, 
\end{equation}
$$ \sum_{g\geq 0} c_g t^{2g}=\left(\dfrac{\frac{t}{2}}{\sin{\frac{t}{2}}} \right) \log \left(\dfrac{\frac{t}{2}}{\sin{\frac{t}{2}}} \right). $$
\end{theorem}
This theorem has direct application to the Gromov-Witten theory to multiple cover formula for Calabi-Yau-3-folds \cite{getzler1998virasoro} \cite{eguchi1997quantum, bryan2000enumerative}. 

    \begin{equation*}
 b_g=
\begin{cases}
	\text{1}    & \text{if  $g=0$;}\\
	\text{$\int_{\bar{\mathcal{M}}_{g,1}}\psi^{2g-2}_1 \lambda_{g} $}   & \text{ if $g> 0$}.
\end{cases}
 \end{equation*}
Now we shall focus on the integral $b_g$, we put this as:
\begin{equation}
    b_g=\frac{(2-2^{2g})}{2^{2g}}\frac{\mathbf{B}_{2g}}{(2g)!}
\end{equation}
Up to a genus of $6$, we can know the value of $b_g$ such as; $b_1= -\frac{1}{24}$, $b_2= \frac{7}{5760}$, $b_3= -\frac{31}{9676780}$, $b_4= \frac{127}{154828800}$, $b_5= -\frac{73}{3503554560}$ and $b_6=\frac{1414477}{2678117105664000}$.\\

\begin{example}\label{genus}
For genus $g$ we shall check the following left factorial property for $!(2g)=K_{2g}$, where 
$$!(2g)=K_{2g}=\sum_{r=0}^{2g-1} r!= 0! + 1! + 2! + \cdots + (2g-1)!.$$ 
We compute few examples of $!(2g) b_g$ as follows:
\begin{itemize}
 \item $(!0)b_0= b_0=1$
    \item $(!2)b_1= 2\cdotp b_1$
    \item $(!4)b_2= 10\cdotp b_2$
   \item $(!6)b_3= 154\cdotp b_3$
\end{itemize}
\end{example}

\subsection{Genus $g$ and kurepa factorial}

In particular, we can see that
\begin{align*}
    !(2g)=K_{2g}=\sum_{r=0}^{2g-1} r!&= 0! + 1! + 2! + \cdots+(2g-2)! + (2g-1)!\\
    !(2g)=\sum_{r=0}^{2g-1} r!&= 0! + 1! + 2! + \cdots +(2g-2)!+ (2g-1)!\\
    !(2g)&= 0! + 1! + 2! + \cdots +(2g-2)!+ (2g-1)!\\
    !(2g)&=\sum_{r=0}^{2g-2} r!+ (2g-1)!\\
    !(2g)&=!(2g-1) + (2g-1)!\\
    K_{2g}&= K_{2g-1}+ (2g-1)!
\end{align*}
it is easy to see that 
\begin{equation}
    (2g-1)!=  K_{2g}-  K_{2g-1}= !(2g)- !(2g-1)
\end{equation}
one can generalize this recurssion to a lemma as follows:
\begin{lemma}
    For all genus $g>0$, $$(2g-1)!=  K_{2g}-  K_{2g-1}= !(2g) - !(2g-1)$$
    and \cite{kurepa1971left} if $$!(n+1)= !n + n!$$ for all $n\in \BN.$ Then
    $$(2g-(n+1))!=  K_{2g-n}-  K_{2g-(n+1)}= !(2g-n) - !(2g-(n+1))$$
    and 
     $$(2g+n-1)!=  K_{2g+n}-  K_{(2g+n-1)}= !(2g+n) - !(2g+n-1)$$
    for $n\geq 0.$
\end{lemma}
Following this lemma we can compute few arithmetic examples as follows:

\begin{example}
For $(2g-(n+1))!=  K_{2g-n}-  K_{2g-(n+1)}$ we check for $n=0,1,2.$
\begin{itemize}
 \item $(2g-1)!=  K_{2g}-  K_{2g-1}= !(2g) - !(2g-1)$
    \item $(2g-2)!=  K_{2g-1}-  K_{2g-2}= !(2g-1) - !(2g-2)$
    \item $(2g-3)!=  K_{2g-2}-  K_{2g-3}= !(2g-2) - !(2g-3)$
\end{itemize}

Similarly we check for $(2g+n-1)!=  K_{2g+n}-  K_{(2g+n-1)}= !(2g+n) - !(2g+n-1)$
\begin{itemize}
 \item $(2g-1)!=  K_{2g}-  K_{2g-1}= !(2g) - !(2g-1)$
    \item $(2g)!=  K_{2g+1}-  K_{2g}= !(2g+1) - !(2g)$
    \item $(2g+1)!=  K_{2g+2}-  K_{2g+1}= !(2g+2) - !(2g+1)$
   \item $(2g+2)!=  K_{2g+3}-  K_{2g+2}= !(2g+3) - !(2g+2).$
\end{itemize}

\end{example}

\begin{theorem}
    Let  $n\in \BN$ and fix $n\geq 0$, for genus $g=g_1 + g_2$ with $g_1, g_2 >0$ then
   \begin{align*}
    (2(g_1 + g_2)-(n+1))!&=  K_{2(g_1 + g_2)-n}-  K_{2(g_1 + g_2)-(n+1)}\\
    &= !(2(g_1 + g_2)-n) - !(2(g_1 + g_2)-(n+1))
     \end{align*}
    and 
\begin{align*}
    (2(g_1 + g_2)+n-1)!&=  K_{2(g_1 + g_2)+n}-  K_{(2(g_1 + g_2)+n-1)}\\
    &= !(2(g_1 + g_2)+n) - !(2(g_1 + g_2)+n-1).
\end{align*}

\end{theorem}
To prove this, we shall state a corollary which explains the details 

\begin{corollary}
    For genus $g=g_1 + g_2$ with $g_1, g_2 >0$ and
    \begin{align*}
        (2g_1-1)!=  K_{2g_1}-  K_{2g_1-1}= !(2g_1) - !(2g_1-1)\\
        (2g_2-1)!=  K_{2g_2}-  K_{2g_2-1}= !(2g_2) - !(2g_2-1)
    \end{align*}
then
$$  (2g_1-1)! +  (2g_2-1)!=(2(g_1 + g_2)-1)!$$
and 
 $$(2(g_1 + g_2)-1)!=  K_{2(g_1 + g_2)}-  K_{2(g_1 + g_2)-1}= !2(g_1 + g_2) - !(2(g_1 + g_2)-1).$$

\end{corollary}
\begin{proof}
    \begin{align*}
    (2g_1-1)! +  (2g_2-1)!&=  !(2g_1) - !(2g_1-1) +!(2g_2) - !(2g_2-1)  \\    
    &=!(2g_1)+!(2g_2) - !(2g_1-1) - !(2g_2-1)  \\  
    &= !2(g_1 + g_2) - !(2(g_1 + g_2)-1)\\
    &=(2(g_1 + g_2)-1)!
\end{align*}


\begin{align*}
    (2g_1-1)! +  (2g_2-1)! &= K_{2g_1}+ K_{2g_2} -  K_{2g_1-1}-K_{2g_2-1}\\
    &=  K_{2(g_1 + g_2)}-  K_{2(g_1 + g_2)-1}
\end{align*}
\end{proof}

\begin{proposition}
For $ b_g=\frac{(2-2^{2g})}{2^{2g}}\frac{\mathbf{B}_{2g}}{(2g)!}=\frac{\beta_{g}}{2^{2g}}$ and  $H_{2g-1}= \sum_{r=0}^{2g-1}\frac{1}{r}$ the following are satisfied:

\begin{align*}
    (2g-1)!\cdotp b_g&=  K_{2g}\cdotp b_g-  K_{2g-1}\cdotp b_g= !(2g)\cdotp b_g- !(2g-1)\cdotp b_g\\
(2g-1)! H_{2g-1} &=  K_{2g}H_{2g-1} -  K_{2g-1}H_{2g-1}= !(2g)H_{2g-1} - !(2g-1)H_{2g-1}= s(2g,2)\\ 
 (2g-1)!\cdotp b_g H_{2g-1} &=  K_{2g}\cdotp b_g H_{2g-1} -  K_{2g-1}\cdotp b_g H_{2g-1}\\
 &= !(2g)\cdotp b_g H_{2g-1} - !(2g-1)\cdotp b_g H_{2g-1}= s(2g,2)\cdotp b_g
\end{align*}
In particular, for the genus $g=g_1 + g_2$ with $g_1, g_2 >0$ if
\begin{align*}
        (2g_1-1)! \cdotp b_g&=  K_{2g_1}\cdotp b_g-  K_{2g_1-1}\cdotp b_g= !(2g_1)\cdotp b_g - !(2g_1-1)\cdotp b_g\\ \\
        (2g_2-1)!\cdotp b_g&=  K_{2g_2}\cdotp b_g-  K_{2g_2-1}\cdotp b_g= !(2g_2)\cdotp b_g - !(2g_2-1)\cdotp b_g\\ \\
        (2g_1-1)! H_{2g_1-1} &= K_{2g_1}H_{2g_1-1} -  K_{2g_1-1}H_{2g_1-1}\\
    &= !(2g_1)H_{2g_1-1} - !(2g_1-1)H_{2g_1-1}= s(2g_1,2)\\ \\
     (2g_1-1)!\cdotp b_{g_1} H_{2g_1-1} &=  K_{2g_1}\cdotp b_{g_1} H_{2g_1-1} -  K_{2g_1-1}\cdotp b_{g_1} H_{2g_1-1}\\
 &= !(2g_1)\cdotp b_{g_1} H_{2g_1-1} - !(2g_1-1)\cdotp b_{g_1} H_{2g_1-1}= s(2g_1,2)\cdotp b_{g_1} 
    \end{align*}
then
\begin{align}
  &(2g_1-1)!\cdotp b_{g_1} \times  (2g_2-1)! \cdotp b_{g_2}=(2g_1-1)! (2g_2-1)!\cdotp b_{g_1} b_{g_2}\\ \nonumber
  &= \bigg[ !(2g_1)!(2g_2)+!(2g_1-1)!(2g_2-1) - !(2g_1)!(2g_2-1)-!(2g_1-1)!(2g_2) \bigg]b_{g_1} b_{g_2}\\
  &=\bigg[ K_{2g_1}K_{2g_2}+K_{2g_1-1}K_{2g_2-1} - K_{2g_1}K_{2g_2-1}-K_{2g_2}K_{2g_1-1} \bigg]b_{g_1} b_{g_2}\\ \nonumber
  &=\big(K_{2g_1} - K_{2g_1-1}\big)\big(K_{2g_2} - K_{2g_2-1}\big)b_{g_1} b_{g_2}
\end{align}
    
\end{proposition}

\begin{theorem}[Faber-Pandhariponde-Zagier and Gessel-Mikki identities]\cite{gessel1999miki, getzler1998virasor}
    Put $b_g=2-2^{2g}\frac{\mathbf{B}_{2g}}{(2g)!}$
\begin{itemize}
    \item \begin{align*}
        \sum_{n=1}^{g}\frac{2^{2n}}{2n}\frac{\mathbf{B}_{2n}}{(2n)!}b_{g-n}&=\sum_{g_1 + g_2=g}\frac{(K_{2g_1} - K_{2g_1-1})(K_{2g_2} - K_{2g_2-1})}{2}b_{g_1} b_{g_2}
        \\&- (K_{2g}\cdotp b_g H_{2g-1} -  K_{2g-1}\cdotp b_g H_{2g-1})
    \end{align*}

    \item $$\frac{n}{2}\sum_{i=2}^{n-2}\frac{\mathbf{B}_i (\frac{1}{2})}{i} \frac{\mathbf{B}_{n-i} (\frac{1}{2})}{n-i}- \sum_{i=0}^{n-2}\dbinom{n}{i} \mathbf{B}_i \bigg(\frac{1}{2}\bigg)\beta_{n-i}= H_{n-1}\mathbf{B}_n \bigg(\frac{1}{2}\bigg)$$
\end{itemize}
    where $\beta_{n}=(-1)^n \frac{\mathbf{B}_n}{n}$, $(2g-1)!b_g= 2^{2g}\beta_{2g}(\frac{1}{2})$, and $\mathbf{B}_n(\frac{1}{2})=(2^{1-n}-1)\mathbf{B}_n$ \cite{faber1998hodge}.
\end{theorem}

Mimicing Vladimirov Theorem \ref{VK} we pose the following question about the $b_g$:

\begin{question}
For odd primes $p=3,5,\cdots$
\begin{enumerate}
    \item Is
\begin{align}
    \sum_{g=1}^{(p-3)/2}  b_g(K_{2g}-1)  \not \equiv  0\pmod{p}
\end{align}
Does this make the Kurepa conjecture $K_p\not\equiv 0 \pmod{p}$? 
\end{enumerate}
\end{question}

\section{Kaneko`s Heuristic about analogues of Euler constant}\label{sec 7}
Kaneko et al. \cite{kaneko2025finite} examined finite analogues of Euler's constant within the framework of finite multiple zeta values and proposed many possible values based on the notions of a "regularized value of $\zeta(1)$" and the series representations of Euler's constant by Mascheroni and Kluyver utilizing Gregory coefficients. In this section we shall extend the results to the 
Gertsch($\mathbb{G}_{p}$) , Lerch($L_p$) and Agoh-Guiga ($\mathbb{AG}_p$) $\pmod{p}$.
The Euler's constant $\gamma$,~\cite{euler1740progressionibus} 
\[ \gamma=\lim_{n\to\infty} \left(1+\frac12+\frac13+\cdots+\frac1n-\log n\right), \]
is an elegant mysterious mathematical constant, see ~\cite{lagarias2013euler} for more details.
Let $(G_n)_n$ be Gregory coefficients defined by
\begin{equation}\label{eq:gregory} \frac{x}{\log(1+x)} = 1+\sum_{n=1}^\infty G_n x^n.\end{equation}
All $G_n$'s are rational numbers alternating in sign ($(-1)^{n-1}G_n>0$ for $n \geq 1$), first several of them being
\[ G_1 = \frac{1}{2}, \quad G_2 = -\frac{1}{12}, \quad G_3 = \frac{1}{24}, \quad G_4 = -\frac{19}{720}, \quad \ldots \ .\]
Several authors, including F.T. Howard, Jordan, Feng Qi, Agoh, and Dilcher, etc. \cite{qi2014explicit, howard1993norlund, howard1994congruences, agoh2010recurrence} have studied the Gregory coefficients, also known as the Bernoulli number of the second kind. Next we shall follow approach by Kaneko et al. \cite{kaneko2017finite, kaneko2025finite, kanekoZagier2026};
Kaneko et al \cite{kaneko2025finite} gave a heuristic argument to ``justify'' the evidences that the element in $\A$ which corresponds to
$\zeta(k)\bmod \zeta(2)\ZZ_\R$ under the conjectural isomorphism above should be $Z_\A(k)$.
They showed that 
\begin{align}
 \zeta(k)\;\,\text{``}\!\underset{\text{Fermat}}{\equiv}\text{''}\;\zeta(k-(p-1))
\underset{\text{Euler}}{\=}-\frac{\mathbf{B}_{p-k}}{p-k}
\equiv \frac{\mathbf{B}_{p-k}}k\ \pmod p.
\end{align}
``Fermat'' means forcibly applying Fermat's little theorem to each summand of the Riemann zeta value. This can also be interpreted as a forbidden extrapolation of Kummer's congruence, which is valid for negative integers, to positive integers.

Kaneko gave a heuristic explanation to argue that this is an analogue of Euler's constant by recalling that $\gamma$ is a ``regularized value of $\zeta(1)$'', or, 
\[ \gamma=\lim_{s\to1} \left(\zeta(s)-\frac1{s-1}\right).\]
Applying the reasoning 
\[ \zeta(1)\,\text{``}\equiv\text{"}\,\zeta(1-(p-1))=-\frac{\mathbf{B}_{p-1}}{p-1}, \]
ending up in the value which is not $p$-integral. To remedy that, he added a ``polar term'' and have
\[ -\frac{\mathbf{B}_{p-1}}{p-1}+\frac1p\in\Z_{(p)}, \]
by the theorem of Clausen and von Staudt~\cite[Theorem~3.1]{arakawa2014bernoulli}. Define $\beta_n\ (n\ge1)$ by
\[ \beta_n\coloneqq\begin{cases}\mathbf{B}_n/n & \text{if }p-1\nmid n, \\ (\mathbf{B}_n+p^{-1}-1)/n & \text{if } p-1\mid n,\end{cases} \]
the value above is $-\beta_{p-1}$, and notably, it is known that various modulo $p$ congruences for $\beta_n$ hold uniformly, regardless of whether $n$ is divisible by $p-1$ or not (see.~\cite{johnson1975p}). Kaneko observed that the Kummer's congruence $\beta_{2n}\equiv \beta_{2n+p-1}$ ($p\geq 5$) holds for any $n$ and, by the (forbidden) extrapolation, the ``congruence''
\[ \text{``regularized value of }\zeta(1)\text{'' ``}\equiv\text{''} -\beta_{p-1}=-\frac{1}{p-1}\left(\mathbf{B}_{p-1}+\frac1p-1\right)= -\frac{W_p}{p-1}  \mod{p}\]
might look natural
\cite{crandall1997search}.

\subsection{$\gamma$, $W_p$ and Mascheroni $\pmod{p}$}
Kaneko et al. \cite{kaneko2025finite} introduced $\gaW$ as an analogue of $\gamma$ in $\A$.
Further, they introduce other analogues of Euler's constant in $\A$ and establish relations with $\gaW$ and other quantities in $\A$, such as an analogue $\log_\A(x)$ of logarithm.

\begin{definition}\label{Euler-W}
We define $\gaW$ in $\A$ by
\[ \gaW\coloneqq\left(W_p\bmod p\right)_p\in\A. \]
\end{definition}

\begin{definition}  We define another possible analogue $\gaM$ of $\gamma$ in $\A$ by
\[ \gaM\coloneqq\left(\sum_{n=1}^{p-2}\frac{|G_n|}{n}\bmod p\right)_p \in \A. \]
\end{definition}

\begin{definition} For $k\ge2$, define $G_\A(k)\in\A$ by
\[ G_\A(k)\coloneqq\left(G_{p-k}\bmod p\right)_p. \]
\end{definition}

 \begin{definition} Given the Fermat quotient
$q_p(x) = \frac{x^{p-1} - 1}{p} \in\Z_{(p)}$
 We define $\log_\A\colon\Q^{\times}\to \A$ by
\[ \log_\A(x) \coloneqq\left(q_p(x)\bmod p\right)_p\in\A. \]
\end{definition}

Now set 
\[ \ell_\mathcal{A}(x)\coloneqq x\log_\A(x)\in\A. \] 
As can be easily seen (and well known), $\log_\A$ satisfies the functional equation
\begin{equation}\label{eq:loglaw}
\log_\A(xy)=\log_\A(x)+\log_\A(y).
\end{equation}

\begin{Proper}[Kaneko-Matsusaka-Seki \cite{kaneko2025finite}]\label{Kane}
    For $\sum_{m=0}^{p-2} \frac{(-1)^m}{m+1}\equiv2q_p(2)  \pmod{p}$ and $\ker(\log_\A)=\{1,-1\}$ under the ABC-conjecture~\cite{silverman1988wieferich};
\begin{enumerate}
    \item

\begin{theorem}\label{thm:main1} We have
	\[
		\gaM = \gaW + \ell_\mathcal{A}(2) - 1
	\]
    equivalently,
		$$\sum_{n=1}^{p-2} \frac{|G_n|}{n} \equiv W_p + 2 q_p(2) - 1 \pmod{p}.$$
\end{theorem} 

\item 
\begin{theorem} For $k\ge2$, we have
\[ G_\A(k)= (-1)^k\sum_{j=1}^k(-1)^{j-1}\binom{k}{j}\la(j+1). \]
\end{theorem}

\item 
\begin{theorem}
For $k\geq 2$, $ G_\A(k)\neq0$ under the ABC-conjecture.
\end{theorem}


\end{enumerate}

\end{Proper}

\begin{proof}  
For details, see Kaneko et al. \cite{kaneko2025finite}
\end{proof}

\subsection{Gertsch($\mathbb{G}_{p}$) , Lerch($L_p$) and Agoh-Guiga ($\mathbb{AG}_p$) $\pmod{p}$}

We use similar heuristic as Kaneko et al. \cite{kaneko2025finite} and define an analogue $\gaG$ of Euler's constant in $\A$ as follows.
It is well known that,
\begin{align*}
 K_p=\sum_{ n=0}^{p-1} n!&= 0! + 1! + 2! + \cdots + (p-2)! + (p-1)!\\
        K_p&= \sum_{ n=0}^{p-2} n! + (p-1)!=\mathbf{Bell}_{p-1} + (p-1)! \\
         K_p&\equiv \mathbf{Bell}_{p-1} -1 \pmod{p}.
\end{align*}

Let 
$$\mathbb{G}_{p} = \frac{(K_p - \mathbf{Bell}_{p-1}) + 1}{p}= \frac{(\mathbf{Der}_{p-1} - \mathbf{Bell}_{p-1}) + 1}{p}\in\Z $$
be the Gertsch quotient [see definition \ref{quotient}].

\begin{definition}\label{gq}  We define $\gaG$ in $\A$ by
\[ \gaG\coloneqq\left(\mathbb{G}_p\bmod p\right)_p\in\A. \]
\end{definition}

\begin{lemma}\label{Nzero}
    $\gaG\neq 0$.
\end{lemma}
\begin{proof}
    There are no known Gertsch primes, that is, the Gertsch quotient, $\mathbb{G}_p \not\equiv 0 \pmod{p}$, primes such that $\mathbb{G}_p$ is not divisible by $p$.
\end{proof}

\begin{proposition}\label{lerch}[Lerch]\cite{lerch1905theorie}
For an odd prime p, we have
    $$ W_p=\sum_{a=1}^{p-1} q_p(a) \pmod{p}.$$
\end{proposition}
Now we see from \cite{kaneko2025finite}
\begin{equation}\label{L-wil}
    \gaW=\sum_{x=1}^{p-1} \log_\A(x) 
\end{equation}
from the definition of the Lerch quotient \cite{sondow2014lerch} 
\begin{align*}
    L_p&=\dfrac{\sum_{a=1}^{p-1} q_p(a)- W_p}{p}
\end{align*}
    
\begin{definition}  We define $\gaL$ in $\A$ by
\[ \gaL\coloneqq\left(L_p\bmod p\right)_p\in\A. \]
\end{definition}

Also, from the Agoh-Guiga conjecture we define analogues $\gaAG$ and $\gaQ$ of Euler's constant in $\A$ as follows. From lemma \ref{Alem}

   $$ \mathbb{AG}_p=\frac{p\mathbf{B}_{p-1}+1}{p}$$
and $$ \mathrm{Q}_p(m):= \mathbb{AG}_p + q_p(m) \in \Z_p$$
then 

\begin{definition}\label{QAG}  We define $\gaAG$ and $\gaQ$ in $\A$ by
\[ \gaAG\coloneqq\left(\mathbb{AG}_p\bmod p\right)_p\in\A. \]
and 
\[ \gaQ(m)\coloneqq\left(\mathrm{Q}_p(m)\bmod p\right)_p=\gaAG + \log_\A(m)  \]
where $$ \left(\mathrm{Q}_p(m)\bmod p\right)_p\in\A.$$
\end{definition}

\begin{theorem}\label{KA}
    For an odd prime $p$ and an integer $m\geq 1$ with $p\nmid m$, we have 
  \begin{enumerate}
      \item $$\sum_{k=1}^{p-2}  \frac{1}{m^k}\frac{\mathbf{B}_k}{k!}\equiv W_p +q_p(m) \pmod{p}=\gaW + \log_\A(m)=\gaQ(m)-1 $$
      
\item $$\sum_{k=1}^{p-2}  \frac{(-1)^k}{m^k}\frac{\mathbf{B}_k}{k!}\equiv W_p +q_p(m)+ \frac{1}{m} \pmod{p}=\gaW + \log_\A(m)+ \frac{1}{m} $$
where $\log_\A(p-m) \coloneqq\left(q_p(p-m)\bmod p\right)_p\in\A$ and $q_p(p-m)\equiv q_p(m)+ \frac{1}{m} \pmod{p}$ and $\gaQ(m)\coloneqq\left(\mathrm{Q}_p(m)\bmod p\right)_p=\gaAG + \log_\A(m). $

  \item specifically if $m=2$
  $$\sum_{k=1}^{p-2}  \frac{1}{2^k}\frac{\mathbf{B}_k}{k!}\equiv \mathrm{Q}_p(2)-1\equiv W_p +q_p(2) \pmod{p}=\gaW + \log_\A(2) $$

\item $$\sum_{k=1}^{p-2}  \frac{(-1)^k}{2^k}\frac{\mathbf{B}_k}{k!}\equiv W_p +q_p(2)+ \frac{1}{2} \pmod{p}=\gaW + \log_\A(2)+ \frac{1}{2} $$


\item specifically if $m=1$, 
$$\sum_{k=1}^{p-2}  \frac{\mathbf{B}_k}{k!}= \gaW= \gaQ(1)-1 $$

\item $$\sum_{k=1}^{p-2}  (-1)^k\frac{\mathbf{B}_k}{k!}\equiv \gaW  + 1 \pmod{p}$$

\item For any fixed integer $n$ with $1\leq n \leq p-1$, we have
$$\sum_{k=1}^{p-2}  H_n^{(k)}\frac{\mathbf{B}_k}{k!}\equiv nW_p +q_p(n!) \pmod{p}=n\gaW + \log_\A(n!) $$
where the generalized harmonic number of order $k\geq 1$, $H_n^{(k)}:= \sum_{m=1}^{n}\frac{1}{m^k} $ for $n\geq 1$ and $\gaQ(1)\coloneqq\left(\mathrm{Q}_p(1)\bmod p\right)_p=\gaAG + \log_\A$. 
      
  \end{enumerate} 
  
\end{theorem}
\begin{proof}
    The proof of this follows from properties \ref{prop all}, Theorem \ref{Ag}, and property \ref{Kane}.
\end{proof}

\begin{corollary}
From $\gaM = \gaW + \ell_\mathcal{A}(2) - 1$, the following identities are equivalent:
    \begin{equation}
		\sum_{n=1}^{p-2} \frac{|G_n|}{n}\equiv \sum_{k=1}^{p-2}  \frac{1}{2^k}\frac{\mathbf{B}_k}{k!} +  q_p(2) - 1 \pmod{p} 
	\end{equation}
    equivalently,
    \begin{equation}
		-G_{p-1} - \frac{1}{p} \equiv \sum_{k=1}^{p-2}  \frac{1}{2^k}\frac{\mathbf{B}_k}{k!} +  q_p(2) - 1 \pmod{p}. 
	\end{equation}
    
\end{corollary}
\begin{proof}
    The proof of this is straightforward by using the theorem of Kaneko et al. \cite{kaneko2025finite} from property \ref{Kane} and theorem \ref{KA} above.
\end{proof}


\begin{lemma}\label{Kanswer}
If $ V_p \equiv  V_{p}^{*} +1 \pmod{p}$,
      then numbers $V_{p}\equiv  W_p + 2\pmod{p}$ and $V_{p}^{'}\equiv  W_p + \frac{1}{2} \pmod{p}$ is equivalent to  $V_{p}= \gaW+ 2\in \A$ and  $V_{p}^{'}= \gaW+ \frac{1}{2}\in \A$ respectively.
\end{lemma}
\begin{proof}
    The proof of this is trivial and left to the reader as an exercise.
\end{proof}


\subsection{Proof of theorem \ref{atta}}
\begin{proof}
\begin{align*}
      K_p= \sum_{ n=0}^{p-2} n! + (p-1)!&=\mathbf{Bell}_{p-1} + (p-1)! \\
       K_p + 1&=\mathbf{Bell}_{p-1} + (p-1)! +1\\
       \frac{ K_p + 1}{p}&\equiv\frac{\mathbf{Bell}_{p-1}}{p} + \frac{(p-1)! +1}{p} \pmod{p}\\
        \frac{K_p + 1}{p}&\!\underset{\text{Wilson}}{\equiv} \frac{ \mathbf{Bell}_{p-1}}{p}+ W_p \pmod{p}\\
        \frac{K_p + 1}{p}-\frac{ \mathbf{Bell}_{p-1}}{p}&\equiv W_p  \pmod{p}\\
         \frac{K_p - \mathbf{Bell}_{p-1} + 1}{p}&\!\underset{\text{Gertsch quotient}}{\equiv} W_p \pmod{p}\\
         \mathbb{G}_p &\equiv  W_p \pmod{p} \quad\mbox{see Theorem \ref{G-Wil} and definition \ref{quotient}}
\end{align*}
Also, from proposition \ref{lerch},  $\sum_{a=1}^{p-1} q_p(a) =W_p \pmod{p}$ and equation \ref{KBW}, we have
\begin{align*}
\frac{K_p + 1}{p}&\!\underset{\text{Lerch}}{\equiv}\frac{ \mathbf{Bell}_{p-1}}{p}+ \sum_{x=1}^{p-1} q_p(x)\\
\frac{K_p + 1}{p}-\frac{ \mathbf{Bell}_{p-1}}{p}&\equiv \sum_{x=1}^{p-1} q_p(x) \pmod{p}\\
 \frac{K_p - \mathbf{Bell}_{p-1} + 1}{p}&\!\underset{\text{Gertsch quotient}}{\equiv}\sum_{x=1}^{p-1} q_p(x) \pmod{p}\\
  \mathbb{G}_p &\equiv \sum_{x=1}^{p-1} q_p(x) \pmod{p}
\end{align*}
it is easy to see that $\mathcal{H}_p \equiv \frac{\sum_{x=1}^{p-1} q_p(x)- \mathbb{G}_p}{p} \pmod{p} $, if we set $x=2$, this yields
\begin{align*}
\frac{K_p - \mathbf{Bell}_{p-1} + 1}{p}&\!\underset{\text{Gertsch quotient}}{\equiv}\sum_{x=1}^{p-1} q_p(2) \pmod{p}\\
\frac{K_p - \mathbf{Bell}_{p-1} + 1}{p}&\!\underset{\text{Fermat quotient}}{\equiv}(p-1) q_p(2) \pmod{p}\\
\frac{K_p - \mathbf{Bell}_{p-1} + 1}{p}&\!\underset{\text{Wieferich}}{\equiv}-q_p(2) \pmod{p}
\end{align*}
which yields equation \ref{KBF}. Similarly, if $x=3$ we yield \ref{KBM}
$$\frac{K_p + 1}{p}\!\underset{\text{Mirimanoff}}{\equiv}\frac{ \mathbf{Bell}_{p-1}}{p}-q_p(3).$$

We know from Kaneko et al, that for $k\geq 2$ an element $Z_\A(k)\in \A$ is given by
$$Z_\A(k)\coloneqq\left(\frac{\mathbf{B}_{p-k}}k\,\bmod p\right)_p.$$
Also, by definition \ref{gq}, $\left(\mathbb{G}_p\bmod p\right)_p\in\A$, this implies that the Bell modulo $p$ must be in the "poor man's adele ring" (defined by Kaneko and Zagier see \cite{kanekoZagier2026}) $$\left(\frac{ \mathbf{Bell}_{p-1}}{p}\bmod p\right)_p\subset\left(\mathbb{G}_p\bmod p\right)_p\in\A$$ and the Kurepa modulo $p$ must also be in 

$$\left( \frac{K_p + 1}{p}\bmod p\right)_p\subset \left(\mathbb{G}_p\bmod p\right)_p\in\A$$
from Theorem \ref{G-Wil}, $\mathbb{G}_p \equiv  W_p \pmod{p}$ and from definition \ref{gq}, and $\left(\mathbb{G}_p\bmod p\right)_p$ thus,
\begin{align*}
    \left(\mathbb{G}_p\bmod p\right)_p=\left(W_p\bmod p\right)_p\\
    \left( \frac{K_p - \mathbf{Bell}_{p-1} + 1}{p}\bmod p\right)_p=\left(W_p\bmod p\right)_p\\
     \left(  \left( \frac{K_p + 1}{p}-\frac{ \mathbf{Bell}_{p-1}}{p}\right)\bmod p\right)_p=\left(W_p\bmod p\right)_p\\
 \left( \frac{K_p + 1}{p}\bmod p\right)_p- \left( \frac{ \mathbf{Bell}_{p-1}}{p}\bmod p\right)_p=\left(W_p\bmod p\right)_p\\
  \left( \frac{K_p + 1}{p}\bmod p\right)_p- \left( \frac{ \mathbf{Bell}_{p-1}}{p}\bmod p\right)_p=\gaW.
\end{align*}
Next, from equation \ref{KBL} we know
$$\frac{K_p + 1}{p}\!\underset{\text{Lerch}}{\equiv}\frac{ \mathbf{Bell}_{p-1}}{p}+ \sum_{x=1}^{p-1} q_p(x).$$
One easily sees that 
\begin{align*}
     \left(  \left( \frac{K_p + 1}{p}-\frac{ \mathbf{Bell}_{p-1}}{p}\right)\bmod p\right)_p=\left(\sum_{x=1}^{p-1} q_p(x)\bmod p\right)_p\\ 
      \left( \frac{K_p + 1}{p}\bmod p\right)_p- \left( \frac{ \mathbf{Bell}_{p-1}}{p}\bmod p\right)_p=\left(\sum_{x=1}^{p-1} q_p(x)\bmod p\right)_p\\
      \left( \frac{K_p + 1}{p}\bmod p\right)_p- \left( \frac{ \mathbf{Bell}_{p-1}}{p}\bmod p\right)_p=\sum_{x=1}^{p-1}\log_\A(x) \\
      \left( \frac{K_p + 1}{p}\bmod p\right)_p=\left( \frac{ \mathbf{Bell}_{p-1}}{p}\bmod p\right)_p+\sum_{x=1}^{p-1}\log_\A(x), 
\end{align*}
we know from equation \ref{KBF}
$$\frac{K_p + 1}{p}\!\underset{\text{Fermat}}{\equiv}\frac{ \mathbf{Bell}_{p-1}}{p}-q_p(2)$$ similarly, 

\begin{align*}
 \left( \frac{K_p + 1}{p}\bmod p\right)_p=\left( \frac{ \mathbf{Bell}_{p-1}}{p}\bmod p\right)_p+\sum_{x=1}^{p-1}\log_\A(2) \\
 \left(   \frac{K_p + 1}{p} =\frac{ \mathbf{Bell}_{p-1}}{p}+\sum_{x=1}^{p-1}\log_\A(2) \bmod p\right)_p\\
     \left(   \frac{K_p + 1}{p} =\frac{ \mathbf{Bell}_{p-1}}{p}-\log_\A(2) \bmod p\right)_p\\
    \mbox{and if we set $x=3$}\\
 \left( \frac{K_p + 1}{p}\bmod p\right)_p \!\underset{\text{Mirimanoff}}{=}\left(\frac{ \mathbf{Bell}_{p-1}}{p}\bmod p\right)_p -\log_\A(3)
\end{align*}
\end{proof}

\subsection{Proof of Theorem \ref{Kaneko-ring}}
\begin{proof}
From Theorems \ref{G-Wil}, \ref{KBW} and definition \ref{gq} we can easily see that
\begin{align*}
     \left(  \left( \frac{K_p + 1}{p}-\frac{ \mathbf{Bell}_{p-1}}{p}\right)\bmod p\right)_p=\left(W_p\bmod p\right)_p\\ 
      \left( \frac{K_p + 1}{p}\bmod p\right)_p- \left( \frac{ \mathbf{Bell}_{p-1}}{p}\bmod p\right)_p=\left(\frac{(p-1)! +1}{p}\bmod p\right)_p\\
      \left(   \frac{K_p + 1}{p} =\frac{ \mathbf{Bell}_{p-1}}{p}+\frac{(p-1)! +1}{p}\bmod p\right)_p\\
      \end{align*}
It is well known that $ K_p= \sum_{ n=0}^{p-2} n! + (p-1)!=\mathbf{Bell}_{p-1} + (p-1)! $ where from Wilson's theorem, $(p-1)!\equiv -1 \bmod{p}$ thus,
   \begin{align*}   
          \left(   K_p \equiv \mathbf{Bell}_{p-1} -1 \pmod p\right)_p
\end{align*}
furthermore, Gertsch \cite{gertsch1999congruences} showed that
 $$\mathbf{Der}_{p-1}\equiv K_p \pmod{p}\equiv \mathbf{Bell}_{p-1} - 1\pmod{p}.$$
 This leads to
 $$\left(   K_p \bmod p\right)_p=\gaKp$$
 and Sun and Zagier \cite{sun2011curious} showed that 
  $$(S_m)_{(p)}=\sum_{0<k<p-1} \dfrac{\mathbf{Bell}_k}{(-m)^k} \pmod{p}\equiv (-1)^{m-1}\mathbf{Der}_{m-1}\pmod{p}$$
also, Kaneko and Zagier \cite{kanekoZagier2026}(see example 6) showed that this will always belong to the subring $\BZ$ of the poor man's adele ring $\A,$ hence
 $$\gaKp\coloneqq\left(   K_p \bmod p\right)_p=\left(   \mathbf{Der}_{p-1} \bmod p\right)_p=\left(   \mathbf{Bell}_{p-1} - 1\pmod{p}\right)_p\in \A.$$
 Now what remains is to show that $\gaKp\neq0$.
 We shall rephrase this question as follows;
 \begin{question}
     Is there any prime $p\geq 3$ such that $\frac{K_p + 1}{p}\equiv 0  \mod{p}$?\\  Equivalently, does  
     $\bigg(\frac{ \mathbf{Bell}_{p-1}}{p}+ W_p\bigg)\equiv 0 \mod{p}$ for any $p\geq 3$?
 \end{question}
 
 To do this, we shall show from equation \ref{KBW}, $$ \frac{K_p + 1}{p}\!\underset{\text{Wilson}}{\equiv} \frac{ \mathbf{Bell}_{p-1}}{p}+ W_p$$ 
 that is non-zero, that is, $\frac{K_p + 1}{p} \bmod{p}$ does not vanish,
 It is well known that the only prime $p$ where $W_p\equiv 0 \pmod{p}$ are $p=5, 13, 563.$ Similarly, we check with same prime $p=5, 13, 563$ for $$\frac{ \mathbf{Bell}_{p-1}}{p} \pmod{p}$$ using mathematical software such as PARI-GP, we see that $\frac{ \mathbf{Bell}_{p-1}}{p} \pmod{p}$ never vanishes at $p= 13, 563$, therefore, $\frac{ \mathbf{Bell}_{p-1}}{p}+ W_p \pmod{p}$ will not vanish for any odd prime $p$,
 which makes $\frac{K_p + 1}{p} \bmod{p}$ non-zero. Furthermore, we know from Theorem \ref{atta} that
 $$\frac{K_p - \mathbf{Bell}_{p-1} + 1}{p}\equiv W_p \pmod{p},$$
 from  Theorem \ref{G-Wil} and definition \ref{quotient}, this yields the Gertsch quotient modulo prime
         $$\mathbb{G}_p \equiv  W_p \pmod{p} $$
 and the residue of the Gertsch quotient coincides with the Wilson quotient when $p=3,7, \ldots$(see theorem \ref{G-Wil}). The elements of $\mathbb{G}_p$ are always integers and  $\mathbb{G}_p\mod{p}$ is non-zero (see \ref{wilson} and \ref{intg}), however,  $W_p\mod{p}$ vanishes at $p=5, 13, 563$, but at these same primes the element of $\mathbb{G}_p\mod{p}$ remains non-zero. 
 Since  $\frac{ \mathbf{Bell}_{p-1}}{p}+ W_p \pmod{p}$ is non-zero, $\frac{K_p + 1}{p} \bmod{p}$ is also non-zero and there are no known odd prime for which $\mathbb{G}_p\equiv 0\pmod{p}$ (see Theorem \ref{NP}) then $K_p \pmod p$ must not vanish for $p\geq 3.$ 
 To finally check for the poor man's adele ring $\A$, we use equation \ref{E-wil};
 \begin{align*}
       \left( \frac{K_p + 1}{p}\bmod p\right)_p &=\left(\frac{ \mathbf{Bell}_{p-1}}{p}\bmod p\right)_p + \gaW\\
        \left( \frac{K_p + 1}{p}\bmod p\right)_p- \left( \frac{ \mathbf{Bell}_{p-1}}{p}\bmod p\right)_p&= \gaW\\
         \left( \frac{K_p - \mathbf{Bell}_{p-1} + 1}{p}\bmod p\right)_p&= \gaW\\
        \left(\mathbb{G}_p\bmod p\right)_p&= \gaW\\
         \gaG&= \gaW.
 \end{align*}
 Kaneko et al. \cite{kaneko2025finite} mentioned that they do not know whether $\gaW$ is zero or not. However, we know that there are no known primes for which $\mathbb{G}_p\equiv 0\pmod{p}$ and also $\gaG\neq 0$ by lemma \ref{Nzero}, it is trivial to see that $\left( \frac{K_p + 1}{p}\bmod p\right)_p\subset \gaG$, and $\left(\frac{ \mathbf{Bell}_{p-1}}{p}\bmod p\right)_p\subset \gaG$, then $$\gaKp\coloneqq\left(   K_p \bmod p\right)_p=\left(   \mathbf{Der}_{p-1} \bmod p\right)_p=\left(   \mathbf{Bell}_{p-1} - 1\pmod{p}\right)_p$$ does not vanish and $\gaKp \neq 0.$
\end{proof}

\begin{proposition}\label{ago}
Let $K_p \equiv \mathbf{Bell}_{p-1} -1 \pmod p$ for $p\geq 3$ and   $p\mathbf{B}_{p-1}\equiv -1 \pmod{p}$ be the Agoh-Giuga conjecture;  then  
\begin{align*}\label{KBA}
    \frac{K_p + 1}{p}\!\underset{\text{Agoh-Guiga}}{\equiv} \frac{ \mathbf{Bell}_{p-1}}{p}+ \mathbb{AG}_p\tag{\textbf{KBA}}
\end{align*}  
and
\begin{align*}\label{GFQ}
\mathbb{G}_p + q_p(m)\underset{\text{Fermat}}{\equiv}  \mathrm{Q}_p(m)\tag{\textbf{GFQ}}
\end{align*}
if we set $m=2$ and $m=3$ then,
\begin{align*}\label{GFW}
\mathbb{G}_p + q_p(2)\underset{\text{Weiferich}}{\equiv}  \mathrm{Q}_p(2)\tag{\textbf{GFW}};
\end{align*} 
\begin{align*}\label{GFM}
\mathbb{G}_p + q_p(3)\underset{\text{Mirimanoff}}{\equiv}  \mathrm{Q}_p(3)\tag{\textbf{GFM}}.
\end{align*}  
From equations \ref{KBA}, \ref{GFQ}, \ref{GFW} and definition \ref{QAG} with $\gaG, \quad \gaQ(m), \quad \log_\A(m)\in \A$, then

\begin{align}
\gaG + \log_\A(m)\coloneqq\gaQ(m)
\end{align}
if we set $m=2$ and $m=3$ then,
\begin{align}
\gaG + \log_\A(2)\coloneqq \gaQ(2)
\end{align}
and 
\begin{align}
\gaG + \log_\A(3)\coloneqq \gaQ(3)
\end{align}
where $\mathrm{Q}_p(m):= \mathbb{AG}_p + q_p(m) $.
\end{proposition}

\begin{proof}
We know that for $p\geq 3$
        $$K_p \equiv \mathbf{Bell}_{p-1} -1 \pmod p$$ 
from definition \ref{KAG} 
\begin{align*}
    K_p &\equiv \mathbf{Bell}_{p-1}+ p\mathbf{B}_{p-1} \pmod{p}\\
     K_p + 1&\equiv \mathbf{Bell}_{p-1} + p\mathbf{B}_{p-1}+1 \pmod{p}\\
       \frac{ K_p + 1}{p}&\equiv\frac{\mathbf{Bell}_{p-1}}{p} + \frac{p\mathbf{B}_{p-1} +1}{p} \pmod{p}\\
        \frac{K_p + 1}{p}&\!\underset{\text{Agoh-Guiga}}{\equiv} \frac{ \mathbf{Bell}_{p-1}}{p}+ \mathbb{AG}_p \pmod{p}\\
\end{align*}
Next, using the Fermat quotient $q_p (m)= \frac{m^{p-1}-1}{p}$ we have 
      \begin{align*}
          \frac{ K_p + 1}{p}+ q_p (m)&\equiv\frac{\mathbf{Bell}_{p-1}}{p} + \frac{p\mathbf{B}_{p-1} +1}{p} + q_p (m) \pmod{p}\\
\frac{K_p + 1}{p}-\frac{ \mathbf{Bell}_{p-1}}{p}+ q_p (m)&\equiv  \frac{p\mathbf{B}_{p-1} +1}{p} + q_p (m) \pmod{p}\\ 
 \frac{K_p - \mathbf{Bell}_{p-1} + 1}{p}+ q_p (m)&\equiv  \frac{p\mathbf{B}_{p-1} +1}{p} + q_p (m) \pmod{p}\\ 
 \mathbb{G}_p + q_p(m)&\equiv  \frac{p\mathbf{B}_{p-1} +1}{p} + q_p (m) \pmod{p}\\ 
  \frac{K_p - \mathbf{Bell}_{p-1} + m^{p-1}}{p}&\equiv \frac{p\mathbf{B}_{p-1}+m^{p-1}}{p}\pmod{p}\\
    \mathbb{G}_p + q_p(m)&\equiv \mathrm{Q}_p(m).
      \end{align*}
Finally, if we set $m=2$ then,
$$\mathbb{G}_p + q_p(2)\equiv \mathrm{Q}_p(2).$$
Next, 
\begin{align*}
    \left(   \mathbb{G}_p \bmod p\right)_p + \left(   q_p(m) \bmod p\right)_p&= \left(   \mathrm{Q}_p(m) \bmod p\right)_p\\
    \gaG + \log_\A(m)&\coloneqq\gaQ(m)
\end{align*}
if we set $m=2$ and $m=3$
      \begin{align*}
    \left(   \mathbb{G}_p \bmod p\right)_p + \left(   q_p(2) \bmod p\right)_p&= \left(   \mathrm{Q}_p(2) \bmod p\right)_p\\
    \gaG + \log_\A(2)&\coloneqq\gaQ(2),
\end{align*}
and 
\begin{align*}
    \left(   \mathbb{G}_p \bmod p\right)_p + \left(   q_p(3) \bmod p\right)_p&= \left(   \mathrm{Q}_p(3) \bmod p\right)_p\\
    \gaG + \log_\A(3)&\coloneqq\gaQ(3).
\end{align*}


\end{proof}

\newpage

\section{Numerical observations and Heuristics of results}\label{sec 8}
In this section we shall discuss certain questions that naturally arise from subsequent sections of this paper.\\
From lemma \ref{Kanswer} we ask;
\begin{question} Is  $$V_{p}= \gaW+ 2=0$$
and also, $$V_{p}^{'}= \gaW+ \frac{1}{2}=0?$$
\end{question}
In a private conversation with Kaneko Masanobu, he helped to check with PARI/GP software for prime values $p\leq 10^6$, in particular, He observed that 
$$V_{p}\equiv  W_p + 2\pmod{p}$$ is congruent to zero modulo $p$ for primes, $p=3, 7, 71$. Also, for 
$$V_{p}^{'}\equiv  W_p + \frac{1}{2}\pmod{p}$$ is congruent to zero modulo $p$ for primes, $p=3, 227, 1163$.

    To better understand Theorem \ref{VK}, we ask the following questions:
    \begin{itemize}
        \item Does $ \sum_{m=1}^{(p-3)/2} \frac{\mathbf{B}_{2m}}{(2m)!}\pmod{p}$ vanish for $p\geq 3$?

\item Similarly, does $ V_p=\sum_{k=0}^{p-2} (-1)^k \frac{\mathbf{B}_k}{k!}\pmod{p} $ vanish?
    \end{itemize}
From the Vladimirov's atoms and Theorem \ref{VK} and Theorem \ref{Francis}, we mimic with the Hodge integral constant $b_g$ discovered by Faber and Pandhariponde \cite{faber1998hodge} to see if we can get any information about the Kurepa modulo prime. 
The following question arose;
\begin{question}
For all odd primes $p=3,5,\cdots$
\begin{enumerate}
    \item is
\begin{align}
    \sum_{g=1}^{(p-3)/2}  b_g(K_{2g}-1)  \not \equiv  0\pmod{p}
\end{align}
Does this make the Kurepa conjecture $K_p\not\equiv 0 \pmod{p}$? 
\end{enumerate}
\end{question}
However, in a private communication with Don Zagier, He pointed out that 
this seems to be unnatural since $\sum_{g=1}^{(p-3)/2}  b_g(K_{2g}-1)$ could not tell us much about the value of the $K_p\bmod{p}.$

\subsection{The congruence $0 \pmod{p}$ }
There are no known primes $p\geq 3$ for which the Gertsch quotient, $\mathbb{G}_p \equiv 0 \pmod{p}.$ 
From Theorem \ref{G-Wil} we know that the Gertsch quotient and Wilson quotient coincide at certain primes say $p=3,7, 2887\ldots$
$$\mathbb{G}_p \equiv W_p \pmod{p} $$
However, at these primes $p=3,7, 2887\ldots$, the Gertsch quotient $\mathbb{G}_p \not\equiv 0 \pmod{p}.$ Also, since $$\mathbb{G}_p \equiv W_p \pmod{p} $$ it is natural to argue that at the prime where $W_p\equiv 0 \mod{p}$ such as $p=5,13,563$, the Gertsch quotient may also be congruent to zero modulo $p$, but this is not the case since  $$\mathbb{G}_p \not\equiv W_p \pmod{p} $$ for $p=5,13,563.$ If such $W_p\equiv 0 \bmod{p}$ existed it will be for $p> 5\times 10^{8}$ or $p\geq 2\times 10^{13}$ see \cite{crandall1997search, costa2014search}.
Another way to naturally argue is to use the Lerch relation; 
$$ \sum_{a=1}^{p-1} q_p(a) \bmod{p} \equiv W_p.$$
Now, remark that it is the 
$$(q_p(1)+q_p(2)+q_p(3)+q_p(4)+\cdots+ q_p(p-1)) \equiv W_p \bmod{p}$$ therefore just as theorem \ref{atta} and theorem \ref{Kaneko-ring}, although
$$\frac{K_p - \mathbf{Bell}_{p-1} + 1}{p}\!\underset{\text{Wieferich}}{\equiv}-q_p(2) \pmod{p}$$
and 
$$\frac{K_p - \mathbf{Bell}_{p-1} + 1}{p}\!\underset{\text{Mirimanoff}}{\equiv} q_p(3) \pmod{p}$$
yield $q_p(2)\equiv 0 \pmod{p}$ and $q_p(3)\equiv 0 \pmod{p}$ for $p=1093, \quad 3511$ and $p=11,\quad 1006003$ respectively \cite{OEIS:A014127, OEIS:A001220}, it does not make $\mathbb{G}_p \equiv 0 \pmod{p}.$ 
We observe that,
\begin{align*}
\mathbb{G}_p \equiv W_p \pmod{p}\\
    \frac{K_p - \mathbf{Bell}_{p-1} + 1}{p}\equiv W_p \pmod{p}\\
    K_p - \mathbf{Bell}_{p-1}\equiv pW_p -1 \pmod{p^2}\\
     K_p - \mathbf{Bell}_{p-1}\equiv (p-1)! \pmod{p^2}
\end{align*}
from the identity 
$p(p + 1)\mathbf{B}_{p-1}\equiv (p-1)! \pmod{p^2}$ (see equation 2.9 in \cite{carlitz1952some}) we obtain the Kurepa-Bell-Bernoulli \ref{KBB} 
\begin{align*}\label{KBB}
      K_p - \mathbf{Bell}_{p-1}\equiv p(p + 1)\mathbf{B}_{p-1}\pmod{p^2}\tag{\textbf{KBB}}
\end{align*}
One could search for primes which make this relation hold.
Also, all these could help us understand the Kurepa-Bell-Wilson \ref{KBW} congruence in section \ref{sec1} (as well as proposition \ref{strong cond}, corollary \ref{st cond}, and corollary \ref{st cnd}).

\subsection{proof of corollary \ref{st cond}}

The proof of this is straightforward from Theorem \ref{atta}, and if a Gertsch prime exists, it must conform to Theorem \ref{G-Wil}.
To the best of our knowledge, the Wilson primes are exclusively identified as $5, 13$, and $563$, which are not Gertsch primes (see to appendix \ref{secA1}).
It is necessary to identify primes \( p \geq 2 \times 10^{13} \) as referenced in \cite{costa2014search} to explore the potential existence of Wilson primes that may validate Theorem \ref{G-Wil}.

\bmhead{Acknowledgements}
Many thanks to Prof. Kaneko Masanobu for useful suggestions and for assisting me with the mathematical software PARI/GP. Also, my sincere gratitude to Prof. Don Zagier for sharing with me the reference \cite{kanekoZagier2026}, a joint paper with Prof. Kaneko Masanobu in preparation.
Thanks to colleagues Alexandre Cesa, R.M. Perez Wheelock, and Kevin W. S. Jung for their support with the typesetting of the appendix and their encouragement.

\newpage


\section*{Appendix}\label{secA1}

\begin{table}[h]
\normalsize
\renewcommand{\arraystretch}{1.3}
\caption{Some values Gertsch quotient $\mathbb{G}_p$ for prime $p\leq 61$ (see \cite{OEIS:A309483} by Amiram Eldar)}
\label{tab:gp_values}
\begin{tabularx}{\textwidth}{@{} l | >{\raggedright\arraybackslash}X @{}}

\toprule
\textbf{p} & \textbf{$\mathbb{G}_p$} \\
\midrule
3 & 1 \\
5 & 4 \\
7 & 96 \\
11 & 356540 \\
13 & 39903286 \\
17 & 1312583081304 \\
19 & 356826497344324 \\
23 & 51202108292508282304 \\
29 & 10903333036235662560405182340 \\
31 & 8851961858819132893480466080328 \\
37 & \seqsplit{10341369256681418109100257759613689061054} \\
41 & \seqsplit{20410983764150196478167108200311379711212644128} \\
43 & \seqsplit{33471988248845076246704814844693140092683344053436} \\
47 & \seqsplit{119680095889593902169611731792572420181399897412939250680} \\
53 & \seqsplit{1551704320329449188553505544936791636242289216222193046404083939884} \\
59 & \seqsplit{40539189508131106145581275089019179146420416222458964156217471022804657603628} \\
61 & \seqsplit{138722328581443601889768771573817998285456423842798757803191931619802810589153798} \\
\bottomrule
\end{tabularx}
\end{table}

\begin{table}[h]
\normalsize
\renewcommand{\arraystretch}{1.3}
\caption{Agoh-Guiga quotient $\mathbb{AG}_p$ for prime $p\leq 97$}
\label{tab:agp_values}
\renewcommand{\arraystretch}{1.5}
\begin{tabularx}{\textwidth}{@{} l | >{\raggedright\arraybackslash}X @{}}
\toprule
p & \textbf{$\mathbb{AG}_p$} \\
\midrule
3 & 1/2 \\
5 & 1/6 \\
7 & 1/6 \\
11 & 1/6 \\
13 & -37/210 \\
17 & -211/30 \\
19 & 2311/42 \\
23 & 37153/6 \\
29 & -818946931/30 \\
31 & 277930363757/422 \\
37 & -\seqsplit{711223555487930419}/51870 \\
41 & -\seqsplit{6367871182840222481}/330 \\
43 & \seqsplit{35351107998094669831}/42 \\
47 & \seqsplit{12690449182849194963361}/6 \\
53 & -\seqsplit{15116334304443206742413679091}/30 \\
59 & \seqsplit{1431925649981017658678758915153153}/6 \\
61 & -\seqsplit{19921854762028779869513196624259348280501}/930930 \\
67 & \seqsplit{21979104807855756030621185500775109585700001}/966 \\
71 & \seqsplit{2120255418779301462015162920814890260724481131}/66 \\
73 & -\seqsplit{79834999474930741238880510200562566647705319203319743}/1919190 \\
79 & \seqsplit{5251219817410137067027582728475216120154422473068360551}/42 \\
83 & \seqsplit{20204989749218624540038006142003251809731759368316306203393}/6 \\
89 & -\seqsplit{14735129086224915820174285663138335318491576130022793756145167309813}/690 \\
97 & -\seqsplit{2181447933992438279356677609379631274979834581330517877841636427010831632473617}/46410 \\
\bottomrule
\end{tabularx}
\end{table}

\clearpage


\renewcommand{\arraystretch}{1.2}
\begin{center}
\begin{longtable}{c | c | c | c}

\caption{Values of $\mathbf{Bell}_{p-1} \bmod p$, Wilson quotient $W_p \pmod p$, and their sum modulo $p\leq 600$ \cite{OEIS:A007540}}
\label{tab:bell_wilson} \\

\toprule
\textbf{$p$} & \textbf{$\mathbf{Bell}_{p-1} \bmod p$} & \textbf{$W_p \pmod p$} & \textbf{Sum mod $p$} \\
\midrule
\endfirsthead

\toprule
\textbf{$p$} & \textbf{$\mathbf{Bell}_{p-1} \bmod p$} & \textbf{$W_p \pmod p$} & \textbf{Sum mod $p$} \\
\midrule
\endhead

\bottomrule
\endfoot

3 & 2 & 1 & Fractional \\
\rowcolor{highlight}
5 & 0 & 0 & 3 \\
\rowcolor{highlight}
7 & 0 & 5 & 6 \\
11 & 2 & 1 & Fractional \\
\rowcolor{highlight}
13 & 11 & 0 & Fractional \\
17 & 14 & 5 & Fractional \\
19 & 10 & 2 & Fractional \\
23 & 22 & 8 & Fractional \\
29 & 18 & 18 & Fractional \\
31 & 3 & 19 & Fractional \\
37 & 6 & 7 & Fractional \\
41 & 5 & 16 & Fractional \\
43 & 17 & 13 & Fractional \\
47 & 19 & 6 & Fractional \\
53 & 14 & 34 & Fractional \\
59 & 29 & 27 & Fractional \\
61 & 23 & 56 & Fractional \\
67 & 66 & 12 & Fractional \\
71 & 69 & 69 & Fractional \\
73 & 56 & 11 & Fractional \\
79 & 21 & 73 & Fractional \\
83 & 28 & 20 & Fractional \\
89 & 77 & 70 & Fractional \\
97 & 81 & 70 & Fractional \\
101 & 14 & 72 & Fractional \\
103 & 51 & 57 & Fractional \\
107 & 44 & 1 & Fractional \\
109 & 66 & 30 & Fractional \\
113 & 110 & 95 & Fractional \\
127 & 57 & 71 & Fractional \\
131 & 82 & 119 & Fractional \\
137 & 94 & 56 & Fractional \\
139 & 135 & 67 & Fractional \\
149 & 83 & 94 & Fractional \\
151 & 11 & 86 & Fractional \\
157 & 132 & 151 & Fractional \\
163 & 5 & 108 & Fractional \\
167 & 31 & 21 & Fractional \\
173 & 105 & 106 & Fractional \\
179 & 30 & 48 & Fractional \\
181 & 171 & 72 & Fractional \\
191 & 105 & 159 & Fractional \\
193 & 166 & 35 & Fractional \\
197 & 10 & 147 & Fractional \\
199 & 123 & 118 & Fractional \\
211 & 131 & 173 & Fractional \\
223 & 43 & 180 & Fractional \\
227 & 226 & 113 & Fractional \\
229 & 51 & 131 & Fractional \\
233 & 70 & 169 & Fractional \\
239 & 13 & 107 & Fractional \\
241 & 129 & 196 & Fractional \\
251 & 61 & 214 & Fractional \\
257 & 148 & 177 & Fractional \\
263 & 53 & 73 & Fractional \\
269 & 17 & 121 & Fractional \\
271 & 57 & 170 & Fractional \\
277 & 8 & 25 & Fractional \\
281 & 219 & 277 & Fractional \\
283 & 155 & 164 & Fractional \\
293 & 265 & 231 & Fractional \\
307 & 199 & 271 & Fractional \\
311 & 49 & 259 & Fractional \\
313 & 300 & 288 & Fractional \\
317 & 206 & 110 & Fractional \\
331 & 252 & 164 & Fractional \\
337 & 102 & 41 & Fractional \\
347 & 135 & 235 & Fractional \\
349 & 344 & 8 & Fractional \\
353 & 76 & 151 & Fractional \\
359 & 91 & 184 & Fractional \\
367 & 183 & 100 & Fractional \\
373 & 3 & 224 & Fractional \\
379 & 74 & 133 & Fractional \\
383 & 102 & 122 & Fractional \\
389 & 99 & 234 & Fractional \\
397 & 153 & 219 & Fractional \\
401 & 184 & 235 & Fractional \\
409 & 385 & 151 & Fractional \\
419 & 119 & 375 & Fractional \\
421 & 307 & 7 & Fractional \\
431 & 166 & 392 & Fractional \\
433 & 154 & 371 & Fractional \\
439 & 227 & 375 & Fractional \\
443 & 183 & 149 & Fractional \\
449 & 166 & 412 & Fractional \\
457 & 182 & 246 & Fractional \\
461 & 421 & 55 & Fractional \\
463 & 42 & 417 & Fractional \\
467 & 4 & 77 & Fractional \\
479 & 284 & 299 & Fractional \\
487 & 258 & 89 & Fractional \\
491 & 236 & 318 & Fractional \\
499 & 131 & 422 & Fractional \\
503 & 246 & 458 & Fractional \\
509 & 369 & 379 & Fractional \\
521 & 165 & 170 & Fractional \\
523 & 513 & 10 & Fractional \\
541 & 280 & 194 & Fractional \\
547 & 169 & 397 & Fractional \\
557 & 391 & 96 & Fractional \\
\rowcolor{highlight}
563 & 107 & 0 & Fractional \\

\end{longtable}
\end{center}

\begin{table}
\normalsize
\renewcommand{\arraystretch}{1.3}   
\caption{ (See MIODRAG Zivković \cite{zivkovic1999number}) The factorizations of $!n - 1,\ n \leq 30$}
\label{tab:factorizations}

\renewcommand{\arraystretch}{1.2}

\begin{tabular}{c | l}
\toprule
$n$ & The factorization of $!n - 1$ \\
\midrule

3  & $3$ \\
4  & $3^2$ \\
5  & $3 \times 11$ \\
6  & $3^2 \times 17$ \\
7  & $3^2 \times 97$ \\
8  & $3^4 \times 73$ \\
9  & $3^2 \times 11 \times 467$ \\
10 & $3^2 \times 131 \times 347$ \\
11 & $3^2 \times 11 \times 40787$ \\
12 & $3^2 \times 11 \times 443987$ \\
13 & $3^2 \times 11^2 \times 23 \times 20879$ \\
14 & $3^2 \times 11 \times 821 \times 83047$ \\
15 & $3^2 \times 11 \times 2789 \times 340183$ \\
16 & $3^2 \times 11 \times 107 \times 509 \times 259949$ \\
17 & $3^2 \times 11 \times 225498914387$ \\
18 & $3^2 \times 11 \times 163 \times 20143 \times 1162943$ \\
19 & $3^2 \times 11 \times 19727 \times 3471827581$ \\
20 & $3^2 \times 11 \times 29 \times 43 \times 1621 \times 641751001$ \\
21 & $3^2 \times 11 \times 53 \times 67 \times 662348503367$ \\
22 & $3^2 \times 11 \times 877 \times 3203 \times 41051 \times 4699727$ \\
23 & $3^2 \times 11 \times 11895484822660898387$ \\
24 & $3^2 \times 11 \times 139 \times 2129333 \times 922459185301$ \\
25 & $3^2 \times 11 \times 37^2 \times 29131483 \times 163992440081$ \\
26 & $3^2 \times 11 \times 454823 \times 519472957 \times 690821017$ \\
27 & $3^2 \times 11 \times 107 \times 173 \times 7823 \times 12227 \times 1281439 \times 1867343$ \\
28 & $3^2 \times 11 \times 431363 \times 2882477797 \times 91865833117$ \\
29 & $3^2 \times 11 \times 191 \times 47793258077 \times 349882390108241$ \\
30 & $3^2 \times 11 \times 37 \times 283 \times 5087 \times 1736655143086866180331$ \\
$\vdots$ &\quad \quad  \quad \quad $\vdots$ \\

\bottomrule
\end{tabular}
\end{table}





\clearpage
\bibliography{sn-bibliography}

\end{document}